\documentclass[11pt,reqno]{amsart}

\makeatletter
\usepackage{amssymb}
\usepackage{amsmath}
\usepackage{latexsym}
\usepackage{amsbsy}
\usepackage{amsfonts}

\def\marginpar#1{\ignorespaces}

\newtheorem{theorem}{Theorem}[section]
\newtheorem{proposition}[theorem]{Proposition}
\newtheorem{lemma}[theorem]{Lemma}

\newtheorem{definition}[theorem]{Definition}

\theoremstyle{definition}
\newtheorem{remark}[theorem]{Remark}

\def\be{\begin{eqnarray}}
\def\ee{\end{eqnarray}}
\def\ben{\begin{eqnarray*}}
\def\een{\end{eqnarray*}}
\numberwithin{equation}{section}

\def\AArm{\fam0 \rm}%
\newdimen\AAdi%
\newbox\AAbo%
\def\AAk#1#2{\setbox\AAbo=\hbox{#2}\AAdi=\wd\AAbo\kern#1\AAdi{}}%

\newcommand{\BBone}{{\ensuremath{{\AArm 1\AAk{-.8}{I}I}}}}

\def\eqref#1{(\ref{#1})}
\def\eqlabel#1{\def\@currentlabel{#1}}

\def\formula#1{\def\@tempa{#1}\let\@tempb\theequation\def\theequation{%
\hbox{#1}}\def\@currentlabel{(\theequation)}$$}
\def\endformula{\leqno\hbox{(\@tempa)}$$\@ignoretrue\let\theequation\@tempb}

\def\given{\hskip5\p@\relax\vrule\@width.4\p@\hskip5\p@\relax}

\newcommand{\open}[1]{%
\par\normalfont\topsep6\p@\@plus6\p@\trivlist\item[\hskip\labelsep\itshape#1%
\@addpunct{.}]\ignorespaces}

\DeclareRobustCommand{\close}[1]{%
  \ifmmode 
  \else \leavevmode\unskip\penalty9999 \hbox{}\nobreak\hfill
  \fi
  \quad\hbox{$#1$}}

\newlength{\toskip}

\def \R {{\mathbb R}}

\def \P {{\mathbb P}}
\def \E {{\mathbb E}}
\def \N {{\mathbb N}}

\def \L {{\mathbb L}}

\def \W {{\mathbb W}}
\def \H {{\mathbb H}}

\def \phi {\varphi}

\newcommand     {\PP}{\mathbb{P}}

\makeatother

\begin{document}
\date{\today}

\title[Competitive and weak cooperative stochastic Lotka-Volterra  ...]{COMPETITIVE OR
WEAK COOPERATIVE STOCHASTIC LOTKA-VOLTERRA SYSTEMS CONDITIONED TO
NON-EXTINCTION}

\author[P. Cattiaux]{\textbf{\quad {Patrick} Cattiaux $^{\clubsuit}$}}

\address{{\bf {Patrick} CATTIAUX},\\ Universit\'e Paul Sabatier
Institut de Math\'ematiques. Laboratoire de Statistique et Probabilit\'es, UMR C 5583\\ 118 route
de Narbonne, F-31062 Toulouse cedex 09.} \email{cattiaux@math.univ-toulouse.fr}

 \author[S. M\'el\'eard]{\textbf{\quad {Sylvie} M\'el\'eard $^{\spadesuit}$}}
\address{{\bf {Sylvie} M\'EL\'EARD},\\ Ecole Polytechnique, CMAP, F- 91128 Palaiseau cedex,
.} \email{meleard@cmapx.polytechnique.fr}

\maketitle
 \begin{center}
 \textsc{$^{\spadesuit}$ Ecole Polytechnique} \quad \textsc{$^{\clubsuit}$ Universit\'e de Toulouse}

 \end{center}

\begin{abstract}
We are interested in the long time  behavior  of a two-type
density-dependent biological population  conditioned to
non-extinction, in both cases of competition or weak cooperation
between the two species. This population is described by a
stochastic Lotka-Volterra system, obtained as limit of
renormalized interacting birth and death processes. The weak
cooperation assumption allows the system  not to blow up. We study
the existence and uniqueness of a quasi-stationary distribution,
that is convergence to equilibrium conditioned to non extinction.
To this aim we generalize in two-dimensions spectral tools
developed for one-dimensional generalized Feller diffusion
processes. The existence proof of a quasi-stationary distribution
is reduced to the one for a $d$-dimensional Kolmogorov diffusion
process under a symmetry assumption. The symmetry we need is
satisfied under a local balance condition relying the ecological
rates. A novelty  is the outlined relation between the uniqueness
of the quasi-stationary distribution  and the ultracontractivity
of the killed semi-group.  By a  comparison between  the killing
rates for the populations of each type and the one of the global
population, we show that the quasi-stationary distribution can be
either supported by individuals of one (the strongest one) type or
 supported by individuals of the two types. We thus
highlight two different long time behaviors depending on the
parameters of the model: either the model exhibits an intermediary
time scale for which only one type (the dominant trait) is
surviving, or there is a positive probability to have coexistence
of the two species.
\end{abstract}

\bigskip

\bigskip

\bigskip

\section{\bf Introduction.}\label{Intro}

 Our aim in this paper is to study the long time behavior of a two
 dimensional
  stochastic Lotka-Volterra process  $Z=(Z^1_t,Z^2_t)_{t\geq 0}$, which  describes the size of
   a two-type
  density dependent population. It generalizes the one-dimensional
  logistic Feller diffusion process introduced in \cite{Eth},
  \cite{Lam05} and whose long time scales have been studied in details in \cite{Cat07}.

  \medskip
  \noindent
   More precisely, let us
consider the coefficients \be \label{coeffts} \gamma_1, \gamma_2
> 0\ ,\ r_1, r_2>0\ ;\ c_{11}, c_{22} >0\ ;\ c_{12}, c_{21} \in
\mathbb{R}.\ee The process $Z$, called Stochastic Lotka-Volterra
process (SLVP),  takes its values in $(\mathbb{R}_+)^2$ and is
solution of the following stochastic differential system: \be
\label{SLVP} dZ^1_t=\sqrt{\gamma_1Z^1_t}dB^1_t + (r_1 Z^1_t -
c_{11} (Z^1_t)^2
- c_{12} Z^1_t Z^2_t)\ dt,\nonumber\\
dZ^2_t=\sqrt{\gamma_2Z^2_t}dB^2_t + (r_2 Z^2_t - c_{21} Z^1_t
Z^2_t  - c_{22} (Z^2_t)^2)\ dt, \label{lotka-volterra}\ee where
$B^1$ and $B^2$ are independent standard Brownian motions
independent of the initial data $Z_0$. The extinction of the
population is modelled by the absorbing state $(0,0)$, and the
mono-type populations  by the absorbing sets $\mathbb{R}_+^*\times
\{0\}$ and $\{0\}\times \mathbb{R}_+^*$.
\bigskip

 \noindent This system \eqref{lotka-volterra} can be obtained as an approximation of
 a renormalized
 two-types birth and death
 process in case of large population and small ecological timescale.
  (The birth and death rates are at the same scale than the initial population size).
  This microscopic point of view has been developed in \cite{Cat07} concerning the
    logistic Feller equation  and can be easily generalized to multi-type models.
The coefficients $r_1$ and $r_2$ are the asymptotic growth rates
of $1$-type's and $2$-type's populations. The positive
coefficients $\gamma_1$ and $\ \gamma_2$  can be interpreted as
demographic parameters describing the ecological timescale. The
coefficients $c_{ij}, i,j =1,2$ represent the pressure felt by an
individual holding type $i$ from an individual with type $j$.  In
our case, the intra-specific interaction rates $c_{11}$ and
$c_{22}$ are assumed to be negative, modelling a logistic
intra-specific competition, while
 the inter-specific  interaction rates given by $c_{12}$ and
 $c_{21}$ can be positive or negative. In case where $c_{12}>0$,
 individuals of type $2$ have a negative influence on individuals
 of type $1$, while in case where $c_{12}<0$, they cooperate.
  Our main results proved in this paper
require two main assumptions. The first one is a symmetry
assumption between the coefficients $\gamma_i$ and $c_{ij}$,
$$c_{12} \gamma_2 =c_{21} \gamma_1,$$
 that we will
call ``balance condition'' (\eqref{balance}). It means that the global rates of influence of each
species on the other one are equal. In particular,  the coefficients $c_{12}$ and $c_{12}$ have
both the same sign.  The second main assumption (\eqref{det}) is required in the cooperative case
(when $c_{12}>0$ and $c_{21}>0$) and is given by $$c_{11}c_{22}-c_{12}c_{21}>0,$$ which compares
the intra-specific to the inter-specific interacting rates. This condition will be called the
``weak
  cooperative case''.

\noindent Because of the quadratic drift terms and of the
degeneracy of the diffusion terms near $0$, the SLVP can blow up
and its existence has to be carefully studied.  We prove the
existence of solutions to \eqref{SLVP}  in the competition case
and in the weak cooperative case. In the first case it results
from a comparison argument with independent one-dimensional
logistic Feller processes. In the general case, the existence of
the process $(Z_t)_t$ and a non blow-up condition are less easy to
prove. If \eqref{balance} holds, a change of variable leads us to
study a Kolmogorov process driven by a Brownian motion and the
existence of the process is proved using  a well chosen Lyapounov
funtion. That can be done
 if  \eqref{det} is satisfied and we  don't know if this condition is also necessary to avoid the
blow-up. Under conditions \eqref{balance} and \eqref{det} we will
also
  prove that $(0,0)$ is an absorbing point and that the process
  converges
almost surely to this point. That means that population goes to
extinction with probability one. We will show this property in two
steps. We will firstly show that the process is attracted by one
of the boundaries $\mathbb{R}_+\times \{0\}$ or $\{0\}\times
\mathbb{R}_+$. Once the process has attained one of them, it
behaves as the logistic Feller stochastic differential equation
and tends almost surely to $(0,0)$ in finite time. Therefore, our
main  interest in this paper is to study the long time behavior of
the process $(Z_t)_t$ conditioned to non-extinction, either in the
competition case, or in the weak cooperative case, under the
balance condition.

 Let us outline that the long time behavior of the SLVP considerably differs from
 the one of the
deterministic Lotka-Volterra system which corresponds to the case
where $\gamma_1=\gamma_2=0$. Indeed, a fine study shows that in
this case,  (cf. Istas \cite{Istas:05}), the point $(0,0)$ is an
unstable equilibrium and there are three possible   non trivial
strongly stable equilibria: either  the remaining population is
totally composed of individuals of type $1$ (the trait $1$ is
dominant), or a similar situation holds for trait $2$ or
co-existence of the two types occurs. In particular,  the
population cannot  goes to extinction.

In this paper, we want to describe  the asymptotic behavior of the
SLVP conditioned to non-extinction, thus generalizing the
one-dimensional case studied in \cite{Cat07}.  This question is of
great importance in Ecology. Although the limited competition
resources entail the extinction
 of the population, the extinction time can be large compared to
 human timescale and certain species may survive  for long periods
  teetering on the brink of extinction before dying out.   A natural
    biological question  is  which type will eventually survive
    conditionally to non-extinction.
 We will show that in
the long-time limit,  the two types do not always disappear at the
same time scale in the population and that a  transient mono-type
state can appear. More precisely, we will give some conditions on
parameters ensuring mono-type transient
   states (preserving a dominant trait in a longer time scale ) or
   coexistence of the two traits. The main tool of our study will
   be spectral theory and the conditions will be obtained by
   comparing the smallest eigenvalues of different killed
   operators. Indeed, conditioning to non-extinction, the process
   can either stay inside the positive quadrant (coexistence of
   traits) or attain one of the boundary (extinction of the other
   trait).
 Let us remark that in a work in progress
\cite{Champ08}, Champagnat and Diaconis are studying a similar
problem for two-types birth-and-death processes conditioned to
non-extinction.

The approach we develop is based on the mathematical notion of
quasi-stationarity (QSD) which has been extensively studied. (See
\cite{Pol} for a regularly updated extensive bibliography,
\cite{Ferri,SE04} for a description of the biological meaning,
\cite{FKMP,Goss,SVJ} for the Markov chain case and \cite{Cat07}
for the logistic Feller one-dimensional diffusion). In the latter,
the proofs are based on spectral theory, and the reference measure
is the natural symmetric measure for the killed process. We will
follow these basic ideas.

In our two-dimensional  SLVP  case, the existence of a symmetric
measure will be  equivalent  to the  balance condition and we are
led to study the Kolmogorov equation obtained by change of
variable.
 Before studying the conditioning to non-extinction, we will in a first step study the
  long time behavior
of the population
   conditioned to the  coexistence of the two types   (the process
  stays in the interior of the quadrangle $(\mathbb{R}_+^*)^2$ as soon as
   coexistence between the two types holds).
   Our theoretical results, generalizing what has be done in \cite{Cat07}
    to any dimension are stated in the  Appendix.
   The arguments implying the existence of a quasi-stationary distribution are mainly similar.
   The novelty will concern the uniqueness of  the quasi-stationary
distribution, which  is shown to be related to the
ultracontractivity of the killed process semi-group. We prove that
the Kolmogorov process associated with the SLVP satisfies this
setting and conclude to the existence and uniqueness of the QSD
for the Kolmogorov process conditioned to coexistence. To deduce a
similar result for the process conditioned to non-extinction we
need to carefully compare the  boundaries hitting times, and the
extinction time. That will give us our main theorem (Theorem
\ref{genial}) on the Kolmogorov system. Let us  summarize here
what does it means coming back to  the stochastic Lotka-Volterra
process.

\begin{theorem} \label{genialSLVP}
1) Under assumptions \eqref{balance} and \eqref{det}, the SLVP is
well defined on $\mathbb{R}_+$ and goes to extinction in finite
time with probability one.

2) The long-time behavior of its law conditioned to non-extinction depends on the starting point
$z$ and is given as follows.

\begin{itemize}
\item For all $z^1>0$, if $z=(z^1,0)$, then for all $A\subset \mathbb{R}_+^* \times \mathbb{R}_+$,
$$\mathbb{P}_{(z^1,0)}(Z_t\in A | T_0>0)= (m^1\otimes \delta_0)(A),$$
where $m^1$ is the unique QSD of the logistic Feller process $\
(Y^1_t)_t$ defined in \eqref{dom}.

\item For all $z^2>0$, if $z=(0,z^2)$, then for all $A\subset \mathbb{R}_+ \times \mathbb{R}_+^*$,
$$\mathbb{P}_{(0,z^2)}(Z_t\in A | T_0>0)= (\delta_0\otimes m^2)(A),$$ where $m^2$ is the unique
QSD of $\ (Y^2_t)_t$ defined in \eqref{dom2}.

 \item There is a unique
quasi-stationary distribution $m$ on
$(\mathbb{R}_+)^2\backslash\{(0,0)\}$,  such that for all
$z=(z^1,z^2)$ with $z^1>0$, $z^2>0$, for all $A\subset
(\mathbb{R}_+)^2\backslash\{(0,0)\}$,
$$\mathbb{P}_z(Z_t\in A | T_0>0)= m(A),$$ where $T_0$ is the
extinction time.

\item
If $\lambda_1$, (resp. $\lambda_{1,1}$, $\lambda_{1,2}$) denotes
  the positive killing rates of the global
 population, (resp. the population of type 1, of type 2), we get

\begin{itemize}
\item \textbf{Competition case}:
$\lambda_1>\lambda_{1,1}+\lambda_{1,2}$ and $m$ is given by \ben
m=\delta_0 \otimes m^2 + m^1\otimes \delta_0.\een Furthermore when
$\lambda_{1,2} > \lambda_{1,1}$ (resp. $<$), $m^1$ (resp. $m^2$)
is equal to $0$.

 \textbf{In other words, the model exhibits an
intermediary time scale  when only one type (the dominant trait)
is surviving.}

\medskip \item \textbf{Weak cooperation case}:
we have two different situations.
\begin{itemize}
\item If $\lambda_1>\lambda_{1,i}$ for $i=1$ or $i=2$, the
conclusion is the same as in the competition case. \item If
$\lambda_1<\lambda_{1,i}$ for $i=1$ and $i=2$, then \ben
m=\delta_0 \otimes m^2 + m^1\otimes \delta_0+ m_D,\een where $m_D$
is proportional to $\nu_1$.

\textbf{We thus have a positive probability to have coexistence of
the two species.}
\end{itemize}
\end{itemize}
\end{itemize}
\end{theorem}

\medskip

Let us remark that our analysis is not  reduced to the
$2$-dimensional case, as it is made clear in the Appendix. All the
machinery is still available in any dimension. However the
explicit conditions on the coefficients are then more difficult to
write down. That is why we restrict ourselves to the
$2$-dimensional setting.

\section{Existence of the SLVP and boundary hitting times}

Let us denote $D=(\mathbb{R}_+^*)^2$. Let us remark that $\partial
D$, $ \mathbb{R}_+\times \{0\}$ and $\{0\}\times \mathbb{R}_+$ are
absorbing sets for the process $(Z_t)_t$, as also $\{(0,0)\}$. We
introduce
 \begin{itemize}
 \item $T_0$: the first hitting time  of $\{(0,0)\}$,
\item $T_1$: the first hitting time of $\mathbb{R}_+\times \{0\}$,
\item $T_2$: the first hitting time of $\{0\}\times \mathbb{R}_+$,
\item $T_{\partial D}$: the first hitting time of $\partial D$ (or
the exit time of $D$).
\end{itemize}

Of course, some of these stopping times are comparable. For
example
$$T_{\partial D}\leq T_1\leq  T_0\ ;\ T_{\partial D}\leq T_2\leq  T_0.
$$ On the other hand, $T_1$ and $T_2$ are not directly
comparable.

\bigskip

Let us prove the existence of the SLVP in some cases.
\begin{proposition} \label{exi-compet} If $c_{12}>0$ and $c_{21}>0$,
then there is no blow-up and the process $(Z_t)_t$ is well defined
on $\mathbb{R}_+$. In addition, for all $x\in (\mathbb{R}_+)^2$,
$$\mathbb{P}_x(T_0<+\infty)=1$$
and there exists $\lambda>0$ such that
$$\sup_{x\in (\mathbb{R}_+)^2}\mathbb{E}_x(e^{\lambda T_0})
<+\infty.$$
\end{proposition}

\begin{proof}
In this competition case, the existence of the SLVP is easy to
show, by using a comparison argument (cf. Ikeda-Watanabe \cite{IW}
Chapter 6 Thm 1.1), and the population process $(Z_t)_t$ does not
blow up. Indeed, the coordinate $(Z^1_t)_t$, resp. $(Z^2_t)_t$)
can be upper-bounded by the solution of the logistic Feller
equation  \be dY^1_t=\sqrt{\gamma_1Y^1_t}dB^1_t + (r_1 Y^1_t -
c_{11} (Y^1_t)^2 )\ dt,\label{dom}\ee respectively \be
\label{dom2} dY^2_t=\sqrt{\gamma_2Y^2_t}dB^2_t + (r_2 Y^2_t  -
c_{22} (Y^2_t)^2)\ dt.\ee These one-dimensional processes have
been introduced in \cite{Eth, Lam05} and studied in details in
\cite{Cat07}. It's easy to deduce (by stochastic domination) that
the processes $Z^1$ and $Z^2$ become extinct in finite time.

 The a.s.
finiteness of each $T_i$, hence of $T_{\partial D}$, thus follows.
It has also been shown in \cite{Cat07} that the absorption times
from infinity have exponential moments.

\end{proof}

Let us now consider the general case. We  reduce the problem  by a
change of variables.

\medskip

Let us  define
$(X_t^1,X_t^2)=(2\sqrt{Z_t^1/\gamma_1},2\sqrt{Z_t^2/\gamma_2})$.
We obtain via It\^o's formula

\begin{eqnarray}\label{eqfell3}
dX^1_t & = & dB_t^1 \, + \, \left(\frac {r_1 X_t^1}{2} \, - \,
\frac{c_{11} \gamma_1 (X_t^1)^3}{8} \, -  \, \frac{c_{12} \gamma_2
\, X_t^1 \, (X_t^2)^2}{8} \, -
\frac{1}{2 X_t^1}\right) \, dt \\
dX^2_t & = & dB_t^2 \, + \, \left(\frac {r_2 X_t^2}{2} \, - \,
\frac{c_{22} \gamma_2 (X_t^2)^3}{8} \, -  \, \frac{c_{21} \gamma_1
\, X_t^2 \, (X_t^1)^2}{8} \, - \frac{1}{2 X_t^2}\right) \, dt
 \, . \nonumber
\end{eqnarray}

\bigskip
 In all the following, we will focus on the symmetric case where $X$ is a
Kolmogorov diffusion, that is a Brownian motion with a drift in
gradient form as \be \label{kolmo} dX_t = dB_t \, - \, \nabla V
(X_t) dt.\ee Indeed, all the results of spectral theory obtained
in the Appendix have been stated for such processes. Obvious
computation shows that it requires the following balance condition
on the coefficients:
\begin{equation}\label{balance}
 c_{12} \, \gamma_2 = c_{21} \,
\gamma_1.
\end{equation}
This relation is a symmetry  assumption
  between the global interaction rate
of type $2$ on type $1$ and of type $1$ on type $2$. (Recall that
the  coefficients $\gamma_i$ describe the ecological timescales).
If \eqref{balance} holds, the coefficients $c_{12}$ and $c_{21}$
have the same sign, allowing inter-species competition ($c_{12}$
and $c_{21} >0$) or inter-species cooperation ($c_{12}$ and
$c_{21} <0$).
\medskip

Under this condition,  the process  $X$ is called the stochastic
Lotka-Volterra Kolmogorov process (SLVKP). The potential $V$ is
then equal to
\begin{equation}\label{eqv}
V(x^1,x^2) = \frac 12 \, \sum_{i=1,2} \, \left(\log (x^i) +
\frac{c_{ii}\gamma_i (x^i)^4}{16} - \frac{r_i (x^i)^2}{2}\right)
\, + \, \alpha (x^1)^2  (x^2)^2,
\end{equation}
where \be \alpha =  \frac{c_{12} \, \gamma_2}{16} = \frac{c_{21}
\, \gamma_1}{16}. \ee

\bigskip

 Let us prove the existence of the SLVKP using the
$\mathbb{L}^2$-norm as a Lyapunov function, under a weak
cooperative assumption, that is  if \be \label{det} \alpha <0
\quad \hbox{ and } \quad c_{11}c_{22}-c_{12}c_{21}>0. \ee

\begin{theorem}
\label{exi} Assume  balance condition \eqref{balance}, weak
cooperative assumption \eqref{det}, then there is no blow-up and
the processes $(X_t)$, and then $(Z_t)$, are well defined on
$\mathbb{R}_+$.

 In addition, for all $x\in D$, \ben \P_x(T_{\partial
 D}<+\infty)=1,\een
for both $X$ and $Z$.
\end{theorem}
Hence, under the assumptions of Theorem \ref{exi}, Hypothesis (H1)
of Definition \ref{defhypo1} in Appendix is satisfied, what we
shall use later.
\begin{proof}
 Let us compute
\ben d\|X\|^2_t&=&
d(X^1_t)^2 + d(X^2_t)^2\\
&=& 2(X^1_t dB^1_t + X^2_t dB^2_t) +\sum_{i=1}^2 (X^i_t)^2
\left(r_i-{c_{ii}\gamma_i (X^i_t)^2\over 4} -{c_{ij}\gamma_j
(X^j_t)^2\over 4} \right)dt. \een The quartic function appearing
in the drift term is thus $ - q((x^1)^2,(x^2)^2)$, with
$$q(u,v)=c_{11}\gamma_1 (u)^2 + c_{22}\gamma_2 (v)^2+ 32 \alpha
uv.$$  Decomposing $$q(u,v)=c_{11}\gamma_1 \left(u+{16 \alpha\over
c_{11}\gamma_1} v\right)^2 +{v^2\over c_{11}\gamma_1}
(c_{11}c_{22}-c_{12}c_{21})\gamma_1\gamma_2 \, ,$$ and since
$\alpha < 0$, a necessary and sufficient condition for $q(u,v)$ to
be positive on the first quadrant ($u>0,v>0$), and to go to
infinity at infinity,  is thus
 $c_{11}c_{22}-c_{12}c_{21}
>0 \,$.
 Hence the drift term in the previous
S.D.E. is negative at infinity. It easily follows that
 \ben \sup_{t\in
\mathbb{R}_+} \mathbb{E}(\|X_t\|^2)<+\infty, \een ensuring that
the processes $(X_t)$, and then $(Z_t)$ are well defined on
$\mathbb{R}_+$.

\medskip \noindent
Let us now study the hitting time of the boundary $\partial D$. We
will compare $(X_t^1,X_t^2)$ with the solution of
\begin{eqnarray}\label{eqfell3bis}
dU^1_t & = & dB_t^1 \, + \, \left(\frac {r_1 U_t^1}{2} \, - \,
\frac{c_{11} \gamma_1 (U_t^1)^3}{8}
\, -  \, \frac{c_{12} \gamma_2 \, U_t^1 \, (U_t^2)^2}{8} \right) \, dt \\
dU^2_t & = & dB_t^2 \, + \, \left(\frac {r_2 U_t^2}{2} \, - \,
\frac{c_{22} \gamma_2 (U_t^2)^3}{8} \, -  \, \frac{c_{21} \gamma_1
\, U_t^2 \, (U_t^1)^2}{8} \right) \, dt
 \, . \nonumber
\end{eqnarray}
 Assuming \eqref{balance} and \eqref{det}, the diffusion process ($U_t^1,U_t^2)$ exists and is
unique in the strong sense, starting from any point.  We consider
now the solution built with the same Brownian motions as for $X$.

We shall see that, starting from the same $(x^1,x^2)$ in the first
quadrant,  and for all $t<T_{\partial D}$, $X_t^1\leq U_t^1$ and
$X_t^2 \leq U_t^2$.

To this end we can make the following elementary reasoning. Fix
$\omega$ and some $t<T_{\partial D}(\omega)$. Let us define $s
\rightarrow W_s^i=X_s^i - U_s^i$ for $i=1,2$ and for $s \leq t$.
Of course $W_0^i=0$. Due to the continuity of the paths, $s
\mapsto W_s^i$ is of $C^1$ class and $W^i$ solves an ordinary
differential equation such that $\frac{d}{ds}
(W_s^i)_{|_{s=0}}=-{1\over 2 x^i}<0$.

 Moreover we remark that if at some time $u\leq t$, $W^i_u=0$, then $\frac{d}{ds}
(W_s^i)_{|_{s=u}}=-{1\over 2 X^i_ u}<0$. It follows that $W_s^i
\leq 0$ for $0<s<t$, yielding the desired comparison result.

Denote $S_{\partial D}$ the hitting time  of $\partial D$ for the
process $U$. We thus have $S_{\partial D} \geq T_{\partial D}$. It
is thus enough to show that $S_{\partial D}$ is almost surely
finite. But, remark that $d\nu=e^{-Q(x)} dx$, with \be \label{Q}
Q(x^1,x^2)= \, \sum_{i=1,2} \, \left( \frac{c_{ii}\gamma_i
(x^i)^4}{16} - \frac{r_i (x^i)^2}{2}\right) \, + \, 2 \, \alpha
(x^1)^2 (x^2)^2,\ee
is an invariant (actually symmetric) bounded
measure for $U$. The process $U$  is thus positive recurrent and
it follows that, starting from any point in the first quadrant
$D$, $S_{\partial D}$ is a.s. finite.
\end{proof}

\bigskip
 These
comparison arguments allow us to obtain other interesting
properties of the process $X$, that we collect in the next
proposition

\begin{proposition}\label{propbord}
Under the assumptions of Theorem \ref{exi}, the following holds
\begin{itemize}
\item there exists $\lambda >0$ such that $\sup_{x \in D} \,
\E_x(e^{\lambda \, T_{\partial D}}) < +\infty$, \item for all
$x\in D$, $\P_x(T_{\partial D}=T_i)>0$ for $i=1,2$ (recall that
$T_i$ defined at the beginning of section 2 is the hitting time of
each half axis), and $P_x(T_{\partial D}=T_0)=0$ (recall that
$T_0$ is the hitting time of the origin).
\end{itemize}
\end{proposition}
\begin{proof}
The first point is an immediate consequence of the same moment
controls for both $Y$ and $U$, thanks to the comparison property.
We already mentioned this property for $Y$ in Proposition
\ref{exi-compet}.

The same holds for $U$ since $U$ is known to be ultra-contractive
(see definition \ref{ultracontractivity} in Appendix B2) under the
hypotheses of theorem \ref{exi}. The ultracontractivity property
follows from the fact that the invariant measure $e^{-Q(x)} \, dx$
of the process $U$, ($Q$ defined in \eqref{Q}), satisfies the
conditions of Corollary 5.7.12 of \cite{Wbook}.

One sometimes says that $X$ satisfies the \textit{escape
condition} from $D$.

In order to show the second point we shall introduce another
process.  Namely define for $i=1,2$, $H^i$ as the solution of the
following stochastic differential equation
\begin{equation}\label{eqH}
dH^i_t  =  dB_t^i \, + \, \left(\frac {r_i H_t^i}{2} \, - \,
\frac{c_{ii} \gamma_i (H_t^i)^3}{8} \, -  \, \frac{1}{2
H_t^i}\right) \, dt \, .
\end{equation}
Of course $H^1$ and $H^2$ are independent processes, defined
respectively up to the hitting time of the origin. We decide to
stick $H^i$ in $0$ after it hits 0, as for $X^i$.

On the canonical space $\Omega_t=C([0,t],\bar{D})$ we denote by
$\mathcal P^X$ and $\mathcal P^H$ the laws of the processes
$(X_{s\wedge T_{\partial D}})_{s\leq t})$ and $(H_{s\wedge
T_{\partial D}})_{s\leq t})$ starting from the same initial point
$x$ in $D$. We claim that $\mathcal P^X$ and $\mathcal P^H$ are
equivalent. This is a consequence of an extended version of
Girsanov theory as shown  in \cite{Cat07} Proposition 2.2.  One
then have that for any bounded Borel function $F$ defined on
$\Omega_t$,
$$\E^X \left[F(\omega) \, \BBone_{t<T_{\partial D}(\omega)}\right] \, =
\, \E^H \left[ F(\omega) \, \BBone_{t<T_{\partial D}(\omega)} \,
e^{A(t)}\right]$$
 where
 \ben A(t)&=& \alpha \, (\omega_t^1)^2
(\omega_t^2)^2 - \alpha (x_1)^2 (x_2)^2  - \,  \alpha \, \int_0^t
\, \bigg( 2\alpha \, (\omega_s^1)^2 (\omega_s^2)^2 \,
((\omega_s^1)^2 + (\omega_s^2)^2) - ((\omega_s^1)^2 +
(\omega_s^2)^2)\\
&&- (r_1+r_2) \, (\omega_s^1)^2 (\omega_s^2)^2 + \frac 14
(c_{11}\gamma_1+ c_{22}\gamma_2) \, (\omega_s^1)^2 (\omega_s^2)^2
\, ((\omega_s^1)^2 + (\omega_s^2)^2) + ((\omega_s^1)^2 +
(\omega_s^2)^2)\bigg) ds \, ,\een and $\E^H$ (resp. $\E^X$)
denotes the expectation w.r.t. to $\mathcal P^H$ (resp. $\mathcal
P^X$). Remark that $A(t\wedge T_{\partial D})$ is well defined, so
that the previous relation remains true without the
$\BBone_{t<T_{\partial D}(\omega)}$ replacing $A(t)$ by $A(t\wedge
T_{\partial D})$, i.e. $$\E^X \left[F(\omega) \right] \, = \, \E^H
\left[ F(\omega) \, e^{A(t\wedge T_{\partial D})}\right] \, .$$
This shows the claimed equivalence.

It thus remains to prove the second part of the proposition for
the independent pair $(H^1,H^2)$. But as shown in \cite{Cat07}
Proposition 2.2 again, for each $i=1,2$, $$\E^{H^i}
\left[F(\omega) \, \BBone_{t<T_0(\omega)}\right] \, = \, \E^W
\left[ F(\omega) \, \BBone_{t<T_0(\omega)} \, e^{B(t)}\right]$$
for some almost surely finite $B(t)$, $\E^W$ being the expectation
with respect to the Wiener measure starting at $x^i$. It follows
that the $\mathcal P^{H^i}$ law of $T_0$ is equivalent to the
Lebesgue measure on $]0,+\infty[$, since the same holds for the
$\mathcal P^W$ law of $T_0$. The $\mathcal P^H$ law of $(T_1,T_2)$
is thus equivalent to the Lebesgue measure on $]0,+\infty[ \otimes
]0,+\infty[$ yielding the desired result for $H$ hence for $X$.
\end{proof}

\section{\bf  Existence and Uniqueness of the Quasi-Stationary Distribution for the Absorbing Set
 $\partial D$}

We can define a quasi-stationarity notion associated with each
absorbing set $O$, $\partial D$, $\mathbb{R}_+\times \{0\}$,
$\{0\} \times \mathbb{R}_+$. The results developed in Section 3
will refer to the absorbing set $\partial D$. Its complementary in
$(\mathbb{R}_+)^2$ is $D$, which is an open connected subset of
$\mathbb{R}^2$, conversely to the complementary of other absorbing
sets.

\bigskip
Let us recall what a quasi-stationary distribution  is.
 If $F$ denotes an absorbing set for the $(\mathbb{R}_+)^2$-valued process $Z$  and $T_F$ the hitting time of this set,
  a quasi-stationary
distribution (in short QSD) for $Z$ and for this absorption event
 is a probability measure $\nu$ satisfying
\begin{equation}
 \PP_{\nu}(Z_t\in A \mid T_F>t) = \nu(A),
\end{equation}
 for any
 Borel set $A\subseteq \mathbb{R_+}^2\setminus \{F\}$ and $t\geq 0$.
A specific quasi-stationary distribution is defined, if it exists,
as the limiting law, as $t\to \infty$, of $Z_t$ conditioned on
$T_F>t$, when starting from a fixed population. That is, if  for a
all $x\in \mathbb{R_+}^2\setminus \{F\}$, the limit
$$
\lim_{t\to \infty}\PP_x(Z_t\in A \mid T_F>t)
$$
exists and is independant of $x$, and  defines a probability
distribution on $\mathbb{R_+}^2\setminus \{F\}$, then it is a QSD
called quasi-limiting distribution, or (as we will do here) Yaglom
limit.

It is thus well known, (see \cite{FKMP}), that there exists
$\lambda_F>0$ such that
$$
\mathbb{P}_{\nu}(T_F>t) = e^{-\lambda_F t}.$$ This killing rate
$\lambda_F$ gives the velocity at which the process issued from
the $\nu$-distribution get to be absorbed.

\bigskip

Our aim is now to study the asymptotic behavior of the law of
 $X_t$ conditioned on not reaching the boundary.  All the material
 we need will be developed in Appendix. The later essentially extends to
 higher dimension the corpus of tools introduced in \cite{Cat07}.
  The spectral theory is developed for any Kolmogorov diffusion in
$\mathbb{R}^d$, and for any dimension $d$. Then the existence of a
quasi-stationary distribution is obtained. The novelty is  the
uniqueness result, since it is deduced from the ultracontractivity
of the semigroup.

\medskip
 We will extensively refer to this Appendix,
to study the problem of existence of a quasi-stationary
distribution for the SLVKP, with the potential $V$ defined in
\eqref{eqv}.

We introduce the reference measure, given by
$$\mu(dx^1,dx^2) = e^{-2V(x)} dx = \frac{1}{x^1 \, x^2} \, e^{- Q(x^1,x^2)} \, dx^1 \, dx^2$$
where $Q$ is the symmetric polynomial of degree 4 given in
\eqref{Q}.

It is the natural measure to deal with, since it  makes the
transition semi-group symmetric in $\mathbb{L}^2(\mu)$. One
problem we have to face is that the measure $\mu$ has an infinite
mass due to the behavior of its density near the axes. Remark that
the density is integrable far from the axes, which is equivalent
to
 \be \label{int} \int_D e^{-Q(x)} dx <+\infty. \ee
This property is  immediate in  the competition case ($c_{ii}$,
$\gamma_i$ and $\alpha $ are positive) and has already been shown
in the proof of Theorem \ref{exi}, when \eqref{det} holds.

\medskip

 \begin{proposition}
 \label{below} Assume \eqref{balance} and \eqref{det}. Then, there
  exists some $C>0$ such that for all
 $x\in D$,
 \be \label{fund}  |\nabla V|^2(x)-\triangle V(x) \geq - C.\ee
 We deduce that the
 semigroup  $P_t$ of the SLVKP  killed at time $T_{\partial D}$,
has a density with respect to $\mu$,  belonging to
$\mathbb{L}^2(d\mu)$.
\end{proposition}
\begin{proof}
 Under \eqref{balance} and \eqref{det}, we have seen that the explosion time $\xi$ is
 infinite and thus \eqref{eqh1} is satisfied. Then the conclusion of Theorem \ref{thmstructure}
 holds:
 one
proves by Girsanov's theorem that  the semigroup $P_t$ of the
SLVKP has a density with respect to $\mu$.

 Moreover,
computation gives  \ben &&|\nabla V|^2(x)-\triangle V(x)=\\
&&\bigg(-{1\over 2x^1} + r_1{x^1\over 2} + c_{11}
{\gamma_1(x^1)^3\over 8} +2 \alpha x^1(x^2)^2\bigg)^2 +
\bigg(-{1\over 2x^2} + r_2{x^2\over 2} + c_{22}
{\gamma_2(x^2)^3\over 8} +2 \alpha (x^1)^2x^2\bigg)^2\\
&&+ {1\over 2(x^1)^2} + {r_1\over 2} - c_{11} {3
\gamma_1(x^1)^2\over 8} -2 \alpha  (x^2)^2 + {1\over 2(x^2)^2} +
{r_2\over 2} - c_{22} {3 \gamma_2(x^2)^2\over 8} -2 \alpha
(x^1)^2. \een

Hence, we observe that the terms in $|\nabla V|^2-\triangle V$
 playing a role  near infinity are equal to ${3\over 4}({1\over
x_1^2}+{1\over x_2^2})$ (when one of the coordinates is close to
zero) and to \ben &&\left({c_{11}\gamma_1\over 8} (x^1)^3 +
2\alpha x_1 (x^2)^2\right)^2 +\left({c_{22}\gamma_2\over 8} x_2^3
+ 2\alpha x_2 (x^1)^2\right)^2 \\&& =
(x^1)^2\left(-{c_{11}\gamma_1\over 8} (x^1)^2 + 2\alpha
(x^2)^2\right)^2 + (x^2)^2\left(-{c_{22}\gamma_2\over 8} (x^2)^2 +
2\alpha (x^1)^2\right)^2.\een The two terms in factor of $(x^1)^2$
and $(x^2)^2$ in the first quantity will not be  together equal to
$0$ as soon as $\alpha>0$ or as the determinant of the system \ben
\left\{\begin{array}{ccc}
{c_{11}\gamma_1\over 8} Y_1 + 2\alpha Y_2 &=& 0\\
2\alpha Y_1  + {c_{22}\gamma_2\over 8} Y_2 &=&0.\end{array}\right.
\een is non zero, what is satisfied under the condition
\eqref{det}. It follows that $|\nabla V|^2(x)-\triangle V(x)$
tends to $+\infty$ as $|x|$ tends to infinity. Since $|\nabla
V|^2-\triangle V$ is a smooth function in $D$,  \eqref{fund}
follows.  Then, we deduce from Theorem \ref{thmstructure} that for
each $t>0$, the density of $P_t$ belongs to $\mathbb{L}^2(d\mu)$.
\end{proof}

Let us now state our first main result.

\begin{theorem}\label{thmbranch}
Under Assumptions \eqref{balance} and \eqref{det}, that is if
\begin{itemize}
\item $c_{12} \, \gamma_2 = c_{21} \, \gamma_1$, \item if $\alpha<0$, $ c_{11} c_{22} \, - \,
c_{12} c_{21}
> 0$,
\end{itemize}
 there exists a unique
quasi-stationary distribution $\nu_1$ for the stochastic
Lotka-Volterra Kolmogorov process $X$, which is the quasi-limiting
distribution starting from any initial distribution.

In particular, there exists $\lambda_1>0$ such that for all $x\in
D$, for all $A\subset D$,  \be \label{lim-yag} \lim_{t\to \infty}
e^{\lambda_1 t} \mathbb{P}_x(X_t\in A| T_{\partial D}  >t) =
\nu_1(A). \ee
\end{theorem}

\begin{proof} The proof is deduced from Appendices A and B.
 Standard results on Dirichlet forms (cf. Fukushima \cite{Fuku}) allow us to build a
self-adjoint semigroup on $\mathbb{L}^2(\mu)$, which coincides with $P_t$ for bounded functions
belonging to $\mathbb{L}^2(\mu)$. Its generator  $L$ is non-positive and self-adjoint on
$\mathbb{L}^2(d\mu)$,  with $D(L)\supseteq C_0^{\infty}(D)$. Its restriction to $C_0^{\infty}(D)$
is equal to $$Lg= {1\over 2}\Delta g - V.\nabla g\ ,\quad g\in C_0^{\infty}(D).$$

We now develop a spectral theory for this semigroup (also called
$P_t$), in $\mathbb{L}^2(d\mu)$.

We  check that the hypotheses (H) introduced in Definition
\ref{defhypo1} required to apply  Theorem \ref{thmspectre} are
satisfied under the assumptions \eqref{balance} and \eqref{det}.
Indeed (H4) is obviously satisfied and (H1) and (H2) are deduced
from Proposition \ref{below}. Furthermore, using for example polar
coordinates, one easily shows that  Condition \eqref{det} implies
that
 $$\bar{G}(R)=
\inf\{|\nabla V|^2(x)-\triangle V(x); |x| \geq R \hbox{ and } x\in
D\} \geq c R^6$$ for positive constant $c$. Thus (H3) holds.

 Since Hypotheses $(H)$ and $(H1)$ are satisfied,  Theorem \ref{thmspectre} implies that
 $-L$ has
a purely discrete spectrum of non-negative eigenvalues and  the
smallest one $\lambda_1$ is positive. Let us  prove that the
associated eigenfunction $\eta_1$ belongs to $L^1(d\mu)$. We have
 shown in Section B.2 that $\eta_1 e^{-V}$ is bounded. Thus
\ben \int_D  \eta_1(x) d\mu(x)= \int_D \eta_1(x) e^{-2V(x)}
 dx
\leq \|\eta_1 e^{-V}\|_{\infty} \int_D e^{-V(x)} dx <+\infty,\een
since $e^{-V(x)}= {1\over\sqrt{x^1 x^2}} e^{-{1\over 2}Q(x^1,
x^2)}$.

Thus the eigenfunction $\eta_1$ belongs to $\mathbb{L}^1(d\mu)$
and therefore, as proved in Theorem \ref{thmyaglom}, the
probability measure $\nu_1={\eta_1 d\mu\over \int_D \eta_1 d\mu}$
is the Yaglom limit distribution.

In order to show the uniqueness of the quasi-stationary
distribution, we apply Proposition \ref{propultraunique} relating
this uniqueness property to the ultracontractivity of
 the semi-group $P_t$. Let us
show that  the sufficient conditions ensuring ultracontractivity
stated in Proposition \ref{propkkr} are satisfied by the SLVKP.

The function $V$ is bounded from below in $D_\varepsilon=\{y\in D;
d(y,\partial D)>\varepsilon\}$. In addition, Condition \eqref{det}
implies that $\bar{V}(R)=\sup\{V(x), x\in FD, |x|\leq R\}\leq c'
R^4$ and $\bar{G}(R)\geq c R^6$, for positive constants $c'$ and
$c$. Hence Condition \eqref{ultra} in Proposition \ref{propkkr} is
satisfied with $\gamma_k = k^{-3/2}$. Thus the killed semi-group
$P_t$ of the $SLVKP$ is ultracontractive and then the uniqueness
of the quasi-stationary distribution holds (cf.  Proposition
\ref{propultraunique}).

Hence existence and uniqueness of the quasi-stationary
distribution  holds for $X$.\end{proof}

\begin{remark}
\emph{ Since the laws of $X_t$ and $Z_t$ are related via an
elementary change of variables formula, a similar result will be
true for the stochastic Lotka-Volterra process $Z$.}
\end{remark}

\section{\bf Explicit Quasi-stationary Equilibria - Mono-type
transient states.}

In the above Section we were concerned by  conditioning to
coexistence's event. Let us now come back to our initial question,
that is the long time behavior of the process conditioned to
non-extinction.  The SLVKP dynamics is particular in the sense
that once hitting the boundary $\partial D$, the process will no
more leave it. Hence, for $t\geq T_{\partial D}$ the process will
stay on one half axis, and the dynamics on this axis is given by
the process $(H^i_t)_t$ defined in \eqref{eqH}, that has been
extensively studied in \cite{Cat07}. Thus we  know from
\cite{Cat07} that for $i=1,2$, there is a positive killing rate
$\lambda_{1,i}>0$ and a unique quasi-stationary measure
$\nu_{1,i}$ on the axis $x^j=0$ characterized by the ground state
$\eta_{1,i}$ (eigenfunction related to $\lambda_{1,i}$), which is
a positive function, bounded and square integrable with respect to
the corresponding symmetric measure $\mu^i$ on each axis. More
precisely, we have
$$\nu_{1,i}(dx^i)= \, \eta_{1,i}(x^i) \, \mu^i(dx^i) \, ,$$
with
$$\mu_i(dx^i)=e^{- 2 \int_1^{x^i} q_i(u)du} dx^i \quad \hbox{ and } \quad q_i(u)=
{1\over 2u}-{r_i u\over 2} + { c_{ii} \gamma_i\over 8}.
$$
In addition,  $\forall A \subset \mathbb{R}_+^*$, \be
\label{qsd-unidim} e^{\lambda_{1,i}t} \lim_{t\to \infty}
\mathbb{P}_{x^i} (X^i_t\in A| T_i
>t) = \nu_{1,i}(A),
\ee We deduce from this study that for all $x^1>0$, for all
$A\subset \mathbb{R}_+^*\times \mathbb{R}_+$, (resp. $x^2>0$ and
$A\subset  \mathbb{R}_+\times \mathbb{R}_+^*$),
$$\mathbb{P}_{(x^1,0)}(X_t\in A| T_0>t) =\nu_{1,1}\otimes
\delta_0(A),$$ (resp. $\mathbb{P}_{(0,x^2)}(X_t\in A| T_0>t)
=\delta_0\otimes \nu_{1,2}(A)$).

\medskip

We are now led to study, for $x\in (\mathbb{R}_+^*)^2$ and for $A
\subset (\mathbb{R}^2\backslash \{(0,0)\}$, the asymptotic
behavior of \be \label{decomp}\P_x(X_t \in A \, | \, T_0>t) =
\frac{\P_x(X_t \in A)}{\P_x(T_{\partial D}>t)} \,
 \frac{\P_x(T_{\partial D}>t)}{\P_x(T_0>t)}\ee
where $T_0$ is the hitting time of the origin. We have seen in the
previous section that under \eqref{balance} and  \eqref{det}, the
hitting time  $T_{\partial D}$ is almost surely finite and that
for any $y\in
\partial D$, $T_0$ is $\P_y$ almost surely finite too.
Let us now
  study the asymptotic behavior of $$\frac{\P_x(T_{\partial
D}>t)}{\P_x(T_0>t)} \, .$$

\medskip
We  will compare the three different killing rates $\lambda_1$,
$\lambda_{1,1}$, $\lambda_{1,2}$ corresponding to
 the stopping times $T_{\partial D}$, $T_1$,
 $T_2$, (recall that $T_i$ denotes the hitting time of the axis
$x^j=0$).
\medskip

Notice that if $c_{1,2}=0$, $(X^1,X^2)=(H^1,H^2)$,
$\lambda_1=\lambda_{1,2}+\lambda_{1,1}$, $\nu_1=\nu_{1,1}\otimes
\nu_{1,2}$. Indeed a standard (and elementary) result in QSD
theory says that if $\nu$ is a QSD with absorbing set $C$, the
$\P_\nu$-law of the hitting time $T_C$ of $C$ is an exponential
law with parameter $\lambda$, where $\lambda$ is exactly the
killing rate. In addition, $\P_x(T_C>t)$ behaves like $e^{-\lambda
t}$ for large $t$. Since the minimum of two independent
exponential random variables with parameters $\lambda_{1,2}$ and
$\lambda_{1,1}$ is an exponential variable of parameter
$\lambda_{1,2}+\lambda_{1,1}$, we get
$\lambda_1=\lambda_{1,2}+\lambda_{1,1}$.
\medskip

The following decomposition is the key point to prove our main
theorem. For $x\in (\mathbb{R}_+)^2\backslash \{(0,0)\}$
\begin{equation}\label{eqtzero}
\P_x(T_0>t)= \P_x(T_{\partial D}>t) + \sum_{i=1,2} \,
\P_x\left(T_{\partial D}\leq t \, , \, X^j_{T_{\partial D}}=0 \, ,
\, \P_{(X^i_{T_{\partial D}},0)}(T_0 > t - T_{\partial D})\right)
\end{equation}
where $j=1$ if $i=2$ and conversely. We are interested in the
asymptotic behavior of this quantity, i.e we have to compare the
three  killing rates $\lambda_1$, $\lambda_{1,1}$ and
$\lambda_{1,2}$.

We obtain  our main result.

\begin{theorem}
\label{genial} Under \eqref{balance} and \eqref{det}, there exists
a unique probability measure $m$ such that for all $x\in D$, for
all $A\subset (\mathbb{R}_+)^2\backslash \{(0,0)\}$, \ben
\lim_{t\to \infty} \P_x(X_t \in A \, | \, T_0>t)=\nu(A). \een In
addition, we have
 the following description of $\nu$.

\begin{itemize}
\item Competition case ($c_{12}$ and $c_{21}$ positive). We have
$\lambda_1>\lambda_{1,1}+\lambda_{1,2}$, and the support of the
QSD in included in the boundaries.

Furthermore when $\lambda_{1,2} > \lambda_{1,1}$ (resp. $<$),
$\nu$ is given by $\nu_{1,1}\otimes \delta_0$ (resp. $\delta_0
\otimes \nu_{1,2}$).

In other words, the model exhibits an intermediary time scale for
which only one type (the dominant trait) is surviving.

\medskip

\item Cooperation case ($c_{12}$ and $c_{21}$ negative). We have
two different situations.
\begin{itemize}
\item If $\lambda_1>\lambda_{1,i}$ for $i=1$ or $i=2$, the
conclusion is the same as in the competition case. \item If
$\lambda_1<\lambda_{1,i}$ for $i=1$ and $i=2$, then \be
\label{qsdgenial} \nu=  {1\over 1 + \sum_{i,j=1}^2 {c_j\over
\lambda_{1,i}-\lambda_1}}\left({c_2\over \lambda_{1,1}-\lambda_1}
\nu_{1,1}\otimes \delta_0 + \delta_0 \otimes {c_1\over
\lambda_{1,2}-\lambda_1}\nu_{1,2} + \nu_1\right),\ee where
$$c_j = \mathbb{P}_{\nu_1}(X^j_{T_{\partial D}}=0).$$

We thus have a positive probability to have coexistence of the two
species.
\end{itemize}
\end{itemize}

\end{theorem}

\begin{remark}
The only remaining case is the one where $\lambda_1=\lambda_{1,1}=\lambda_{1,2}$. However the
proof of the theorem indicates that this situation is similar to the competition case, though we
have no rigorous proof of it. In the discrete setting (as claimed in \cite{Champ08}), a fine
analysis of Perron-Frobenius type is in accordance with our guess.
\end{remark}

\begin{proof}

 Our domination arguments allow us to compare the killing
rates :
\medskip

\begin{itemize}
\item \quad {\bf The competition case.} This is the case $c_{12} >
0$. In this case we can show with a similar argument as in the
proof of Theorem \ref{exi} that starting from the same initial
point, $X_t^i \leq H_t^i$ for $i=1,2$. Hence $\lambda_1 \geq
\lambda_{1,1} + \lambda_{1,2}$. Since the killing rates are
positive, it follows that $\lambda_1 > \lambda_{1,i}$ for $i=1,
2$. In particular $\E_x \left[ e^{\lambda_{1,i} T_{\partial
D}}\right] < +\infty$. Hence
$$\liminf_{t \to +\infty} \, e^{\lambda_{1,i} t} \, \P_x\left(T_{\partial D}\leq t \, , \,
X^j_{T_{\partial D}}=0 \, , \, \P_{(X^i_{T_{\partial D}},0)}(T_0 >
t - T_{\partial D})\right) = $$
\begin{eqnarray*}
& = & \liminf_{t \to +\infty} \E_x\left(\BBone_{T_{\partial D}\leq
t} \, \BBone_{X^i_{T_{\partial D}}=0} \, e^{\lambda_{1,i}
T_{\partial D}} \, e^{\lambda_{1,i} (t -T_{\partial
D})}\P_{(0,X^j_{T_{\partial D}})}(T_0 > t - T_{\partial D})\right)
\\ & \geq &  \, \E_x\left(\BBone_{X^i_{T_{\partial D}}=0}
\, \, e^{\lambda_{1,i} T_{\partial D}} \,
\eta_{1,i}(X^i_{T_{\partial D}})\right) > 0 \quad (\textrm{ at
least for one $i$)},
\end{eqnarray*}
according to Fatou's lemma, the positivity of the ground state and
by Proposition \ref{propbord}. It follows that the rate of decay
of $\P_x(T_0>t)$ is at most $e^{- \lambda_{1,i} t}$, while the one
of $P_x(T_{\partial D}>t)$ is $e^{- \lambda_1 t}$, hence as $t \to
+ \infty$, $$\P_x(X_t \in A \, | \, T_0>t) \to 0 \, ,$$ if
$A\subset D$. Hence, the support of the quasi-stationary
distribution will be included in the boundaries. Thanks to
Proposition \ref{propbord}, we know that both terms in the sum
$\sum_{i=1,2}$ in \eqref{eqtzero} are positive, so that the
leading term in the sum will be equivalent to $\P_x(T_0>t)$. If
$\lambda_{1,1}>\lambda_{1,2}$ this leading term is of order
$e^{-\lambda_{1,2}t}$, the proof being exactly the same as before.
The value of the quasi-stationary distribution follows.
\medskip

\item \quad {\bf The weak cooperative case.} This is the case if
$c_{1,2} < 0$. Here a comparison argument gives $\lambda_1 \leq
\lambda_{1,1} + \lambda_{1,2}$. But there is no a priori reason
for $\lambda_1$ to be smaller than $\lambda_{1,i}$. In particular
if $\lambda_1 > \lambda_{1,i}$ for $i=1$ or $2$, we are in the
same situation as in the competition case, and the quasi-limiting
distribution is supported by $\partial D$ because one exits from
$D$ by hitting $x^j=0$ with a positive probability as we mentioned
in proposition \ref{propbord}.
\medskip

It remains to look at the case $\lambda_1 \leq \lambda_{1,i}$ for
$i=1,2$.

Denote by $\psi_j$ the law of $T_{\partial D}$ when the process
exits $D$ by hitting $x^j=0$. Denote by $\zeta_s^i$ the
conditional law of $X^i_{T_{\partial D}}$ knowing $T_{\partial D}
=s$ and $X^j_{T_{\partial D}}=0$. Then \begin{eqnarray*} &&
e^{\lambda_1 t} \, \P_x\left(T_{\partial D}\leq t \, , \,
X^j_{T_{\partial D}}=0 \, , \, \P_{(X^i_{T_{\partial D}},0)}(T_0 >
t - T_{\partial D})\right)\\
& = & \int_0^{+\infty} \, e^{\lambda_1 t} \, \BBone_{s<t} \,
\E_{\zeta^i_s}(\BBone_{T_0>t -s}) \, \psi_j(ds).
\end{eqnarray*}

It has been proved in Corollary 7.9 of \cite{Cat07} that for any
$\lambda<\lambda_{1,i}$,
$$\sup_{\theta^i} \, \E_{\theta^i} [ e^{\lambda T_0^i}] <
+\infty$$ where $\theta^i$ describes the set of all probability
measures on $x^j=0, x^i>0$ and $T_0^i$ is the first hitting time
of $0$ for the one dimensional  logistic Feller diffusion $H^i$.

Hence $\E_{\zeta^i_s}(\BBone_{T_0>t -s}) \leq C \, e^{- \lambda
(t-s)}$ for some universal constant $C$ and all $s$. It follows
that, provided $\lambda_1 < \lambda_{1,i}$ (in which case we
choose $\lambda_1 < \lambda < \lambda_{1,i}$), for all $T>0$,
$$\lim_{t \to +\infty} \, \int_0^{T} \, e^{\lambda_1 t} \,
\BBone_{s<t} \, \E_{\zeta^i_s}(\BBone_{T_0>t -s}) \, \psi_j(ds) =
0 \, .$$


It remains to study $$\lim_{t \to +\infty} \, e^{\lambda_1 t} \,
\P_x\left(t\geq T_{\partial D}> T \, , \, X^j_{T_{\partial D}}=0
\, , \, \P_{(X^i_{T_{\partial D}},0))}(T_0 > t - T_{\partial
D})\right) \, .$$ To this end we first remark that for $T$ large
enough we may replace $\P_x$ by $\P_{\nu_1}$.  Indeed, denote
$$h(T,t,y)=\P_y\left( t-T \geq T_{\partial D} > 0 \, , \,
X^j_{T_{\partial D}}=0 \, , \, \P_{(X^i_{T_{\partial D}},0)}(T_0 >
t -T - T_{\partial D})\right) \, .$$ Since $\nu_1$ is the Yaglom
limit related to $D$ and $0\leq h(T,t,y)\leq 1$ for all $y\in D$
and all $T<t$, for $\varepsilon
>0$ one can find $T$ large enough such that for all $t>T$, $$\P_x\left(t\geq T_{\partial D}> T \,
, \, X^j_{T_{\partial D}}=0 \, , \, \P_{(X^i_{T_{\partial
D}},0)}(T_0
> t - T_{\partial D})\right) = $$
\begin{eqnarray*}
& = & \E_x(\BBone_{T_{\partial D}>T} \, h(T,t,X_T)) \\ &
\approx_\varepsilon & \, \E_{\nu_1}(\BBone_{T_{\partial D}>T} \,
h(T,t,X_T)) \\ & = & \P_{\nu_1}\left(t\geq T_{\partial D}> T \, ,
\, X^j_{T_{\partial D}}=0 \, , \, \P_{(X^i_{T_{\partial
D}},0)}(T_0
> t - T_{\partial D})\right)
\end{eqnarray*}
where $a \approx_\varepsilon b$ means that the ratio $a/b$ satisfies $1-\varepsilon \leq a/b \leq
1+ \varepsilon$.

Now, since $\nu_1$ is a quasi-stationary distribution, starting
from $\nu_1$, the law of $T_{\partial D}$ is the exponential law
with parameter $\lambda_1$, so that $\psi_j(ds)= c_j \, \lambda_1
\, e^{- \lambda_1 s} ds$ where $c_j=\P_{\nu_1} (X^j_{T_{\partial
D}}=0)$ is the exit probability first hitting the half-axis $j$.
In addition the conditional law $\zeta^i_s$ does not depend on
$s$, we shall denote it by $\pi^i$ from now on. This yields
\begin{eqnarray*}&&e^{\lambda_1 t} \, \P_{\nu_1}\left(t\geq T_{\partial D}> T \, , \,
X^j_{T_{\partial D}}=0 \, , \, \P_{(X^i_{T_{\partial D}},0)}(T_0 >
t - T_{\partial D})\right)\\
& = & e^{\lambda_1 t} \, \int_T^t \, c_j \, \lambda_1 \, e^{-\lambda_1 s} \, \P_{\pi^i}(T_0> t-s) \, ds \\
& = & c_j \, \lambda_1 \, \int_T^t \, e^{-(\lambda_{1,i}
-\lambda_1) (t-s)} \, \left(e^{\lambda_{1,i} (t-s)} \,
\P_{\pi^i}(T_0> t-s)\right) \, ds \, .
\end{eqnarray*}

Recall that $\nu_{1,i}$ is the unique Yaglom limit on axis $i$. As
shown in \cite{Cat07}, for any initial law $\pi$, in particular
for $\pi^i$, $\lim_{t \to +\infty} \, e^{\lambda_{1,i} t} \,
\P_{\pi}(T_0> t) = 1$. Using Lebesgue bounded convergence theorem
we thus obtain that for all $T>0$, $$\lim_{t \to +\infty} \,
e^{\lambda_1 t} \, \P_{\nu_1}\left(t\geq T_{\partial D}> T \, , \,
X^j_{T_{\partial D}}=0 \, , \, \P_{(X^i_{T_{\partial D}},0)}(T_0 >
t - T_{\partial D})\right) = \frac{c_j \, \lambda_1}{\lambda_{1,i}
- \lambda_1} \, .
$$

Since the result does not depend upon $T$, we may use our
approximation result of $\P_x$ by $\P_{\nu_1}$ for all
$\varepsilon$ hence finally obtain for all $x$, $$\lim_{t \to
+\infty} \, e^{\lambda_1 t} \, \P_{x}\left(t\geq T_{\partial D}> T
\, , \, X^j_{T_{\partial D}}=0 \, , \, \P_{(X^i_{T_{\partial
D}},0)}(T_0 > t - T_{\partial D})\right) = \frac{c_j \,
\lambda_1}{\lambda_{1,i} - \lambda_1} \,.
$$

Let us now consider $A\subset \mathbb{R}^2\backslash\{(0,0)\}$. To
compute $\mathbb{P}_x(X_t\in A, T_0>t)$, we have to write
$A=(A\cap D)+ (A\cap \mathbb{R}_+^*\times \{0\}) + (A\cap
\{0\}\times \mathbb{R}_+^* )$ and we will compute separately the
three terms. Let us first remark that Since $\nu_1$ is the unique
QSD related to the absorbing set $\partial D$, it is equal to the
Yaglom limit and thus
$$\frac{\P_x(X_t \in A\cap D)}{\P_x(T_{\partial D}>t)} \to_{t\to +\infty} \nu_1(A\cap
D).$$ Let us now write $A_1=A\cap \mathbb{R}_+^*\times \{0\}$.
Thus \ben \frac{\P_x(X_t \in A_1)}{\P_x(T_{\partial D}>t)} =
\frac{\P_x(X_t \in A_1)}{\P_x(T^1>t, T^2=T_{\partial
D})}\frac{\P_x(T^1>t, T^2=T_{\partial D})} {\P_x(T_{\partial
D}>t)}. \een By a similar reasoning as previously, one can prove
that $\frac{\P_x(X_t \in A_1)}{\P_x(T^1>t, T^2=T_{\partial D})}
\to_{t\to +\infty} \nu_{1,1}(A_1)$ and that $\frac{\P_x(T^1>t,
T^2=T_{\partial D})} {\P_x(T_{\partial D}>t)} \to_{t\to +\infty}
{c_2 \lambda_1\over \lambda_{1,1}-\lambda_1}$. A similar result
inverting indices $1$ and $2$ holds for the third term. That
concludes the proof of \eqref{qsdgenial}, in the case where
$\lambda_1<\lambda_{1,i}$, for $i=1, 2$.

\bigskip
This result together with the competition case indicates (but we do not have any rigorous proof of
this) that if $\lambda_1=\lambda_{1,i}$ for some $i$, then the QSD (conditioned to hitting the
origin) is again supported by the axes.
\end{itemize}\end{proof}

\bigskip

The conclusion of this study is partly intuitive. In the
competition case, one of the species will kill the other one
before dying. In the cooperation case, apparently, all can happen,
including coexistence of the two types up to a common time of
extinction. The problem is to know whether all situations for the
killing rates are possible or not.

Here is an heuristic simple argument (to turn it in a rigorous one involves some technicalities):
Fix all coefficients equal say to $1$ except $c_{12}=c_{21}=- c$ with $c>0$. Condition \eqref{det}
reads $0<c<1$. In this situation $\lambda_{12}=\lambda_{11}=\lambda$ and both $\lambda$ and
$\lambda_1$ depend continuously on $c$. When $c \to 0$ (near the independent case) we know that
$\nu$ concentrates on both axes. When $c \to 1$ the rate $\lambda_1$ of return from infinity
decreases to 0 (this is the point to be carefully checked) so that there is some intermediate
value $c_c$ where the phase transition $\lambda=\lambda_1$ occurs.
\bigskip

\appendix
\section{\bf Killed Kolmogorov diffusion processes and their spectral theory.}\label{seckolmo}

Let $D$ be an open connected subset of $\R^d$, $V$ a $C^2$ function defined on $D$. We introduce
the stochastic differential equation
\begin{equation}\label{eqsde}
dX_t = dB_t \, - \, \nabla V(X_t) dt \quad , \quad X_0=x \in D \, ,
\end{equation}
for which a pathwise unique solution exists up to an explosion time $\xi$. The law of the process
starting from $x$ will be denoted by $\P_x$, and for a non-negative measure $\nu$ on $D$ we denote
by $\P_\nu = \int \P_x \, \nu(dx)$.

For a subset $A$ of the closure $\bar D$ of $D$ and for $\varepsilon > 0$ we introduce
\begin{equation}\label{eqhitting}
T_A^\varepsilon \, = \, \inf \, \{ s\geq 0 \, ; \, d(X_s,A)<\varepsilon \} \quad , \quad T_A =
\lim_{\varepsilon \to 0} \, T_A^\varepsilon \, ,
\end{equation}
where $d(.,.)$ denotes the usual euclidean distance. $T_A^\varepsilon$ and $T_A$ are thus stopping
times for the natural filtration. We shall be mainly interested to the cases when $A$ is a subset
of the boundary $\partial D$.

Our first hypothesis is that the process cannot explode unless it reaches the boundary i.e
\begin{equation}\label{eqh1}
 \textrm{for all $x \in D$, $\xi \geq T_{\partial D} \, , \, \P_x$
almost surely.}
\end{equation}
If $D$ is bounded, \eqref{eqh1} is automatically satisfied. If $D$ is not bounded, it is enough to
find some Lyapunov function. We shall make some assumptions later on implying that a specific
function is a Lyapunov function, so the discussion on \eqref{eqh1} will be delayed.
\medskip

The peculiar aspect of gradient drift diffusion process as
\eqref{eqsde} is that the generator $L$ defined  for any function
$g\in C^\infty$ by
\begin{equation}\label{eqgene}
L \, = \, \frac 12 \, \Delta \, - \, \nabla V . \nabla \quad ,
\end{equation}
is symmetric w.r.t. the measure $\mu$ defined by
\begin{equation}\label{eqgene2}
\mu(dx) \, = \, e^{-2V(x)} \, dx \, .
\end{equation}

The following properties of the process can be proved exactly as
Proposition 2.1, Theorem 2.2  and the discussion at the beginning
of Section 3 in \cite{Cat07}.

\begin{theorem}\label{thmstructure}
Assume that \eqref{eqh1} holds. Then there exists a self-adjoint
semi-group $(P_t)$ on $\mathbb{L}^2(\mu)$ such that for all
bounded Borel function $f$,
$$P_t f(x) := \E_x[f(X_t) \, \BBone_{t<T_{\partial D}}].$$
Moreover  for any bounded Borel function $F$ defined on $\Omega =
C([0,t],D)$ it holds
$$\E_x \left[F(\omega) \, \BBone_{t<T_{\partial D}(\omega)}\right] \, =
\, \E^{\W_x} \left[ F(\omega) \, \BBone_{t<T_{\partial D}(\omega)} \, \exp \left( V(x) \, - \,
V(\omega_t) - \, \frac 12 \, \int_0^t \, (|\nabla V|^2 - \Delta V)(\omega_s) ds\right)\right]$$
where $\E^{\W_x}$ denotes the expectation w.r.t. the Wiener measure starting from $x \in D$.
\smallskip

It follows that for all $x\in D$ and all $t>0$ there exists some density $r(t,x,.)$ that verifies
$$P_t f(x)  = \int_D \, f(y) \,
r(t,x,y) \, \mu(dy)$$ for all bounded Borel function $f$.
\smallskip

If in addition there exists some $C>0$ such that $|\nabla
V|^2(y)-\Delta V(y) \geq - C$ for all $y \in D$, then for all
$t>0$ and all $x \in D$, $r(t,x,\cdot)\in \mathbb{L}^2(\mu)$ with
\begin{equation}\label{eql2}
\int_D \, r^2(t,x,y) \, \mu(dy) \, \leq (1/2\pi t)^{\frac d2} \, e^{Ct} \, e^{2V(x)} \, .
\end{equation}
\end{theorem}

\begin{remark}\label{remexplo}
If $V(x) \to +\infty$ as $x \to \infty$ in $D$, the condition
$|\nabla V|^2(y)-\Delta V(y) \geq - C$ for all $y \in D$ is
sufficient for \eqref{eqh1} to hold. Just use Ito's formula with
$V$ as in \cite{Ro99} Theorem 2.2.19. \hfill $\diamondsuit$
\end{remark}

\smallskip

We thus have that for any measurable and compact subset $A \subset
D$ and any $x \in D$,
\begin{eqnarray}\label{eql2bis}
\P_x(X_t \in A \, , \, T_{\partial D}>t) & = & \int \, \P_y(X_{t-1} \in A \, , \, T_{\partial
D}>t-1) \, r(1,x,y) \, \mu(dy)
\\ & = & \int P_{t-1} (\BBone_A)(y) \, r(1,x,y) \, \mu(dy)  \nonumber \\
& = & \int \, \BBone_A(y) \, (P_{t-1}r(1,x,.))(y) \, \mu(dy) \, . \nonumber
\end{eqnarray}
Hence, the long time behavior of the law of the killed process is
completely described by $P_t r(1,.,.)$.

Since both $\BBone_A$ and $r(1,x,.)$ are in $\L^2(\mu)$, the
$\L^2$ spectral theory of $P_t$ is particularly relevant.

\medskip

 This spectral theory can be deduced from the much well
known spectral theory for Schr\"{o}dinger operators, thanks to the
following standard transform:  define for $f \in \L^2(D,dx)$
(resp. $f \in C^\infty_0(D)$)
\begin{equation}\label{eqschro}
\tilde P_t (f) \, = \, e^{-V} \, P_t(f \, e^V) \quad , \quad \tilde L f = \frac 12 \, \Delta f \,
- \, \frac 12 \, (|\nabla V|^2 - \Delta V) \, f \, .
\end{equation}
$\tilde P_t$ is then a self-adjoint strongly continuous semi-group
on $\L^2(D,dx)$ whose generator coincides with $\tilde L$ on
$C^\infty_0(D)$. Notice that $\tilde P_t$ has a stochastic
representation as a Feynman-Kac semi-group, i.e.
\begin{equation}\label{eqfk}
\tilde P_{t}f(x)= \, \E^{\W_x} \left[ f(\omega(t)) \, \BBone_{t<T_{\partial D}} \, \exp\left( - \,
\frac 12 \, \int_0^t \, (|\nabla V|^2 - \Delta V)(\omega_s) ds\right)\right] \, .
\end{equation}

The spectral relationship is clear: if $\lambda$ is some
eigenvalue for $\tilde P_t$ associated to $\psi$ (i.e. $\tilde P_t
(\psi) = e^{- \lambda t} \psi$), it is an eigenvalue of $P_t$
associated with $\eta = e^V \, \psi$ and conversely.
\smallskip

For simplicity we shall impose conditions ensuring that the
spectrum is discrete, in particular reduced to the pure point
spectrum. Necessary and sufficient conditions for this property to
hold have been obtained by Maz'ya and Shubin (see \cite{MS})
extending results by Molchanov in 1953. The criterion is written
in terms of Wiener capacity, hence not very easy to directly read
on the potential $V$. We shall here assume a less general but more
tractable condition taken from the Euclidean case explained in
chapter 3 of Berezin and Shubin \cite{BS}. To this end we now
introduce our main hypotheses:

\begin{definition}\label{defhypo1}
\begin{enumerate}
\item[(1)]  We say that hypothesis (H1) is satisfied if \eqref{eqh1} holds and if for all $x\in
D$,
$$\P_x(T_{\partial D}<+ \infty) = 1 \, .$$  \item[(2)]  Hypothesis (H2) holds  if $$G(y) = |\nabla
V|^2(y)-\Delta V(y) \geq - C
> - \infty$$ for all $y \in D$ . \item[(3)]  Hypothesis (H3) holds if
$$\bar G(R) = \inf \, \{G(y) ; |y|\geq R
\textrm{ and } y\in D\} \to +\infty \textrm{  as } R \to \infty \,
.$$ \item[(4)]  We say that hypothesis (H4) holds if for all $R>0$
one can find an increasing sequence of compact sets $K_n(R)$ such
that the boundary of $K_n(R)\cap \bar D$ is smooth and $\bigcup_n
(K_n(R)\cap \bar D) = \bar B(0,R) \cap \bar D$, where $\bar
B(0,R)$ is the closed Euclidean ball of radius $R$.
\end{enumerate}
\smallskip

For simplicity we say that (H) holds when \eqref{eqh1} and (H2)-(H4) are satisfied.
\end{definition}

We may now state

\begin{theorem}\label{thmspectre}
Assume that (H) is satisfied, then $-L$ has a purely discrete spectrum $0 \leq \lambda_1 <
\lambda_2 < ...$. Each associated eigenspace $E_i$ is finite dimensional. If (H1) holds,
$\lambda_1 >0$.

Furthermore $E_1$ is one dimensional and we may find a (normalized) eigenfunction $\eta_1$ which
is everywhere positive. In particular for all $f,g$ in $\L^2(\mu)$, $$\lim_{t \to +\infty} \,
e^{\lambda_1 t} \, \langle g,P_t f\rangle_\mu \, = \, \langle g,\eta_1\rangle \, \langle
f,\eta_1\rangle \, .$$
\end{theorem}
\begin{proof}
The proof of the first statement is similar to the one of Theorem
3.1 in \cite{BS}, replacing $B(0,R)$ therein by $D \cap B(0,R)$.
Hypothesis (H4) here is useful to show that the embedding
$\H^1(B(0,R) \cap  D) \hookrightarrow \L^2(B(0,R) \cap  D)$ is
compact. Indeed the result is known replacing $B(0,R)$ by $K_n(R)$
(due to the smoothness of the boundary). To get the compactness
result, it is then enough to use a diagonal procedure.

$\lambda_1 \geq 0$ is obvious since $-L$ is a non-negative
operator.

It directly follows from the representation formula in Theorem
\ref{thmstructure} (or \eqref{eqfk}) that the semi-group is
positivity improving (i.e. if $f \geq 0$ and $f \neq 0$, $P_t f
(x) > 0$ for all $x \in D$ and all $t>0$). The proof of the second
statement (non degeneracy of the ground state $\eta_1$) is thus
similar to the one of Theorem 3.4 in \cite{BS}.

Finally, as in \cite{Cat07} section 3, hypothesis (H1) implies
that for $f \in \L^2(\mu)$, $P_t f$ goes to $0$ in $\L^2(\mu)$ as
$t \to +\infty$. This  shows that $\lambda_1 > 0$.
\end{proof}

\begin{remark}\label{remspectre1d}
Contrary to the one dimensional case, for $i\geq 2$ the
eigenspaces are not necessarily one dimensional.
\end{remark}

\medskip

\section{\bf Quasi-stationary distributions and Yaglom limit in $D$.}\label{secqsd}

The aim of this section is to study the asymptotic behavior of the law of $X_t$ conditioned on not
reaching the boundary.

\subsection{\bf The general result.}

The first result is an immediate consequence of the spectral theory.
\begin{proposition}\label{corspectre}
Assume that hypothesis (H) is satisfied. If $A \subset D$ is such that $\BBone_A \in \L^2(\mu)$,
then for all $x \in D$,
$$\lim_{t \to + \infty} \, e^{\lambda_1 t} \, \P_x(X_t \in A \, , \, T_{\partial D}>t) \, = \,
\langle \BBone_A , \eta_1 \rangle \, \eta_1(x) \, .$$

In particular if $\eta_1 \notin \L^1(\mu)$, $\lim_{t \to +\infty} \P_{x}(X_t \in A \, | \,
T_{\partial D}>t) = 0$ .
\end{proposition}
\begin{proof}
Recall \eqref{eql2bis}, i.e. $\P_x(X_t \in A \, , \, T_{\partial D}>t) = \int \, \BBone_A(y) \,
(P_{t-1}r(1,x,.))(y) \, \mu(dy)$. Since both $\BBone_A$ and $r(1,x,.)$ are in $\L^2(\mu)$ we may
apply Theorem \ref{thmspectre} and get
$$\lim_{t \to + \infty} \, e^{\lambda_1 (t-1)} \, \P_x(X_t
\in A \, , \, T_{\partial D}>t) \, = \, \langle \BBone_A , \eta_1
\rangle \,  \langle r(1,x,.) , \eta_1 \rangle\, .$$ Since $\eta_1$
is an eigenfunction it holds $$e^{\lambda_1} \, \langle r(1,x,.) ,
\eta_1 \rangle =  e^{\lambda_1} P_1 \eta_1(x) = \eta_1(x),$$ where
equalities hold in $\L^2(\mu)$. Since $\eta_1$ satisfies $\tilde L
\eta_1 = - \lambda_1 \, \eta_1$ in $D$, standard results in
p.d.e.'s theory show that $\eta_1$ is regular ($C^2$) in $D$,
hence these equalities extend to all $x\in D$. This yields the
first statement.

For the second statement, choose some increasing sequence $D_n$ of compact subsets of $D$, such
that $\bigcup_n D_n = D$. It holds
\begin{equation*}
\P_{x}(X_t \in A \, | \, T_{\partial D}>t)  =  \frac{\P_{x}(X_t \in A \, , \, T_{\partial
D}>t)}{\P_{x}(X_t \in D \, , \, T_{\partial D}>t)}  \leq  \frac{\P_{x}(X_t \in A \, , \,
T_{\partial D}>t)}{\P_{x}(X_t \in D_n \, , \, T_{\partial D}>t)}
\end{equation*}
so that according to what precedes for all $n$, $$\limsup_{t \to + \infty} \, \P_x(X_t \in A \, |
\, T_{\partial D}>t) \, \leq \, \frac{\langle \BBone_A , \eta_1 \rangle}{\langle \BBone_{D_n} ,
\eta_1 \rangle} \, .$$ The infimum over $n$ on the right hand side is equal to $0$ as soon as
$\int_D \eta_1 d\mu = +\infty$, hence the result.
\end{proof}

In view of what precedes, a non trivial behavior of the
conditional law implies that $\eta_1 \in \L^1(\mu)$. Conversely
this property is enough to get the following theorem whose
statement and proof are the same as Theorem 5.2 in \cite{Cat07}.
Observe that the only thing we have to do is to control $P_t
\BBone_A$ for sets $A$ of possible infinite $\mu$ mass.

\begin{theorem}\label{thmyaglom}
Assume that hypothesis (H) holds and that $\eta_1 \in \L^1(\mu)$.

Then $d\nu_1 = \eta_1 d\mu / \int_D \eta_1(y) \mu(dy)$ is a quasi-stationary distribution, namely
for every $t\geq 0$ and any Borel subset $A$ of $D$,
$$
\P_{\nu_1}(X_t \in A \, | \, T_{\partial D}>t) = \nu_1(A) \, .
$$

Also for any $x>0$ and any Borel subset $A$ of $D$,
\begin{eqnarray}
\label{limit} \lim_{t \to +\infty}e^{\lambda_1t} \, \P_x( T_{\partial D}
> t) \, = \, \left(\int_D \eta_1(y) \mu(dy)\right) \, \eta_1(x) \, ,\end{eqnarray}
$$\lim_{t \to +\infty} e^{\lambda_1t}
 \, \P_x(X_t \in A \, , \, T_{\partial D}
> t) \, = \, \left(\int_A \eta_1(y) \mu(dy)\right) \, \eta_1(x) \,.$$ This implies
 that
$$\lim_{t \to +\infty} \, \P_x(X_t \in A \, | \, T_{\partial D}
> t) \, = \, \nu_1(A) \, ,$$
and the probability measure $\nu_1$ is the Yaglom limit distribution.
\end{theorem}
\begin{remark}\label{remyaglomlemma}
The proof of Theorem 5.2 in \cite{Cat07} lies on the following
estimate $$r(t,x,y) \leq C(x) \, e^{- \lambda_1 t} \, \eta_1(y)$$
for all $x,y$ in $D$ ($(0,+\infty)$ in \cite{Cat07}), $t>1$ and
some function $C(x)$. This result is still true here and the proof
based on the Harnack's inequality is similar. \hfill
$\diamondsuit$
\end{remark}
\bigskip

\subsection{\bf Ground state estimates.}

We wish now to give tractable conditions for $\eta_1$ to be in
$\L^1(\mu)$. Of course if $\mu$ is bounded there is nothing to do
since $\eta_1 \in \L^2(\mu)$, so that this subsection is only
interesting for unbounded $\mu$. For simplicity of notation we
assume that the origin $0 \in \bar D^c$ so that if $x\in D$,
$|x|\geq \alpha >0$. The results of this subsection are adapted
from Section 4 in \cite{Cat07}.
\smallskip

Recall that $\eta_1 = e^V \, \psi_1$ where $\psi_1$ is the ground
state of $\tilde L$ (cf. \eqref{eqschro}), i.e. the (positive and
normalized) eigenfunction of $\tilde L$ associated to
$-\lambda_1$. So in order to get some estimates on $\eta_1$ it is
enough to get some estimates on $\psi_1$. Since $\psi_1 =
e^{\lambda_1} \, \tilde P_1(\psi_1)$ it is interesting to prove
contractivity properties for $\tilde P_1$.

Let us first  recall  the definition of ultracontractivity.

\begin{definition}\label{ultracontractivity}
A semi-group of contractions  $(Q_t)_{t\geq 0}$ is said to be
ultracontractive if $Q_{t}$ maps continuously $\mathbb{L}^2(\mu)$
in $\mathbb{L}^\infty(\mu)$ for any $t>0$.

Remark that by duality, and thanks to the symmetry of $\mu$, $Q_t$
also maps continuously $\mathbb{L}^1(\mu)$ to $\mathbb{L}^2(\mu)$.
\end{definition}

\begin{proposition}\label{schroding}
Assume that Hypothesis (H2) and \eqref{eqfk} are satisfied. Then
$\tilde P_t$ is ultracontractive. It follows that $\psi_1=\eta_1
e^{-V}$ is bounded.
\end{proposition}

\begin{proof}
We may compare the fundamental solution (kernel) of $\tilde P_t$
with the one of the Schr\"{o}dinger equation with constant
potential as in \cite{BS} by directly using the representation
\eqref{eqfk}.

Introduce the Dirichlet heat semi-group in $D$ i.e.
\begin{equation}\label{eqheat}
P_t^D f(x) = \E^{\W_x} \left[ f(\omega(t)) \,
\BBone_{t<T_{\partial D}}\right] = \int_D \, f(y) \, p^D_t(x,y) \,
dy \, .
\end{equation}
Hypothesis (H2) and \eqref{eqfk} immediately imply that
\begin{equation}\label{eqcontrolheat}
\tilde p_t(x,y) \, \leq \, e^{Ct/2} \, p_t^D(x,y) \, \leq \,
e^{Ct/2} \, (2\pi t)^{-d/2} \, e^{|x-y|^2/2t}
\end{equation}
where $\tilde p_t$ denotes the (symmetric) kernel of $\tilde P_t$
w.r.t. the Lebesgue measure.

\eqref{eqcontrolheat} shows that $\tilde P_t$ has a bounded kernel
for all $t>0$, and hence that $\tilde P_t$ is ultracontractive. In
particular $\psi_1$ is bounded $\psi_1 = e^{\lambda_1} \, since
\tilde P_1(\psi_1)$, hence $\eta_1 \, e^{-V}$ is bounded. More
generally any eigenfunction $\psi_k$ is bounded.
\end{proof}
\smallskip


From the previous proposition, we deduce that $\eta_1\in
\mathbb{L}^1(\mu)$ as soon as $\int e^{-V(x)}dx<+\infty$. One can
improve this result. Recall that $p^D_1$ denotes the Dirichlet
 heat kernel defined in \eqref{eqheat}. Notice that $\int_D \, p^D_1(x,y) \, dy = \W_x(T_{\partial D} > t)$
  goes to zero as $x$ tends to the boundary.

\begin{proposition}\label{propl11}
Assume that hypothesis (H) is fulfilled. Assume in addition that there exists some $R>0$ such that
the following  is satisfied
\begin{equation}\label{eql1a}
\int_{D\cap \{d(x,\partial D)>R\}} \, e^{-2V(x)} \, dx \, < \, +\infty  \textrm{ and }
\int_{D\cap \{d(x,\partial D)\leq
R\}} \, \left(\int_D \, p^D_1(x,y) \, dy\right) \, e^{-V(x)} \, dx \, < \, +\infty \, .
\end{equation}

\smallskip

Then $\eta_1 \in \L^1(\mu)$. More generally any eigenfunction $\eta_k \in \L^1(\mu)$.
\end{proposition}
\begin{proof}
Since $\eta_k$ is normalized, $$\int_{D\cap \{d(x,\partial D)\}} \, \eta_k(x) \, d\mu \leq \left(\int_{D\cap
\{d(x,\partial D)\}} \, e^{-2V(x)} \, dx\right)^{1/2} < +\infty \, .$$ Now
\begin{eqnarray*}
\int_{D\cap \{d(x,\partial D)\leq R\}} \, \eta_k(x) \, d\mu & = & \int_{D\cap \{d(x,\partial D)\leq R\}} \,
\psi_k(x) \,
e^{-V(x)} \, dx \\ & = & \int_{D\cap \{d(x,\partial D)\leq R\}} \, \left(\int_D \, e^{\lambda_k} \, \psi_k(y)
\, \tilde p_1(x,y) \, dy\right) \, e^{-V(x)} \, dx
\\& \leq & e^{C/2} \, e^{\lambda_k} \, \parallel \psi_k\parallel_\infty \,
\int_{D\cap \{d(x,\partial D)\leq R\}} \,
\left(\int_D \, p^D_1(x,y) \, dy\right) \, e^{-V(x)}
\, dx
\end{eqnarray*}
and the result follows.
\end{proof}
\smallskip

\begin{remark}\label{remtemps}
 Remark that we can replace $p_1^D$ by any $p_s^D$ with $s>0$ in the previous proof.
 \end{remark}

\subsection{\bf Rate of convergence.}

According to the spectral representation we may decompose each
function in $\L^2$. We introduce some notation.
\begin{definition}\label{defspectral}
We denote by $E_2$ the eigenspace associated with $\lambda_2$. We
know that $dim (E_2) = n_2 < +\infty$ and we may choose an
orthonormal basis of $E_2$, $(\eta_{2,1},...,\eta_{2,n_2})$. We
denote by
$pr^\bot$ the orthogonal projection onto the orthogonal of $\R
\eta_1 \oplus E_2$.
\end{definition}

We thus have that
\begin{equation}\label{eqspectrer}
P_{t-1} r(1,x,.) = e^{- \lambda_1 t} \eta_1(x) \, \eta_1 +
e^{-\lambda_2 t} \, \sum_{i=1}^{n_2} \eta_{2,i}(x) \, \eta_{2,i} +
e^{- \lambda_3 (t-1)} \, h(t,x,.)
\end{equation}
where $h(t,x,.)$ is orthogonal to $\R \eta_1 \oplus E_2$ and such that $\parallel
h(t,x,.)\parallel_{\L^2(\mu)} \leq \parallel pr^\bot r(1,x,.)\parallel_{\L^2(\mu)}$. Hence (recall
\eqref{eql2bis}) if $\BBone_A \in \L^2(\mu)$,
\begin{eqnarray}\label{eqrate1}
\P_x(X_t \in A \, , \, T_{\partial D}>t) & = & e^{-\lambda_1 t}
\langle \BBone_A, \eta_1 \rangle \, \eta_1(x) + e^{-\lambda_2 t}
\, \sum_{i=1}^{n_2} \eta_{2,i}(x) \, \langle \BBone_A,
\eta_{2,i}\rangle \\ & + & e^{- \lambda_3 (t-1)} \, \langle
h(t,x,.),\BBone_A \rangle \, . \nonumber
\end{eqnarray}
If we could replace $A$ by $D$ we would obtain an expansion of the
conditional probability $\P_x(X_t \in A \, | \, T_{\partial
D}>t)$. But actually we have
\begin{lemma}\label{lemultra}
If (H) and \eqref{eql1a} are satisfied, $P_1$ is a bounded operator from $\L^\infty(\mu)$ to
$\L^2(\mu)$.
\end{lemma}
Assume firstly the Lemma \ref{lemultra}. Then we may write
$$\P_x(T_{\partial D}>t) = P_t(\BBone_D)(x) =
P_{t-1}(P_1(\BBone_D))(x)$$ with $P_1(\BBone_D) \in \L^2(\mu)$.
Note that $$\langle P_1(\BBone_D),\eta_k\rangle = \int_D \,
e^{-\lambda_k} \, \eta_k \, d\mu$$ since \eqref{eql1a} implies
that each eigenfunction $\eta_k$ is in $\L^1(\mu)$. We thus deduce
the

\begin{proposition}\label{proprate}
If (H) and \eqref{eql1a} are satisfied then for all $x \in D$ and all measurable subset $A \subset
D$ it holds $$\lim_{t \to +\infty} \, e^{(\lambda_2 - \lambda_1)t} \, \left(\P_x(X_t \in A \, | \,
T_{\partial D}>t) - \nu_1(A)\right) \, = \, \qquad \qquad \qquad $$
$$ \qquad \qquad \qquad \frac{\sum_{i=1}^{n_2} \eta_{2,i}(x) \, \left(\langle \BBone_A,
\eta_{2,i}\rangle \langle \BBone_D, \eta_1\rangle - \langle \BBone_D, \eta_{2,i}\rangle \langle
\BBone_A, \eta_1\rangle\right)}{\eta_1(x) \, (\langle \BBone_D, \eta_1\rangle)^2} \, .$$
\end{proposition}
\bigskip

It remains to prove Lemma \ref{lemultra}. To this end let us first
state an upperbound for $\tilde{p}_t$.

\begin{lemma}\label{lembound}
If hypothesis (H) is fulfilled, there exist a constant $M$ and a
non-negative function $B$ satisfying $\lim_{u \to +\infty} B(u) =
+\infty$ such that for any $x,y$ in $D$, $$0<\tilde p_1 (x,y) \leq
M \, e^{-|x-y|^2/4} \, e^{-B(|x|\vee |y|)} \, .$$
\end{lemma}

\begin{proof}
We can obtain an upper bound for $\tilde p_t$, when hypothesis (H)
is fulfilled. To this end, for a non-negative $f$ but this time
$\varepsilon= |x|/2$  we write
\begin{eqnarray*}
\int_D \, f(y) \, \tilde p_t(x,y) \, dy & = & \E^{\W_x} \left[
f(\omega(t)) \, \BBone_{t<\tau_x(\varepsilon)} \,
\BBone_{t<T_{\partial D}} \, \exp\left( - \, \frac 12 \, \int_0^t
\, (|\nabla V|^2 - \Delta V)(\omega_s) ds\right)\right] \\ & + &
\E^{\W_x} \left[ f(\omega(t)) \, \BBone_{T_{\partial D}> t \geq
\tau_x(\varepsilon)} \, \exp\left( - \, \frac 12 \, \int_0^t \,
(|\nabla V|^2 - \Delta V)(\omega_s) ds\right)\right] \\ & \leq &
e^{- t \, \bar G(|x|/2)/2} \, \E^{\W_x} \left[ f(\omega(t)) \,
\BBone_{t<T_{\partial D})}\right] + e^{Ct/2} \, \E^{\W_x} \left[
f(\omega(t)) \, \BBone_{T_{\partial D} > t \geq
\tau_x(\varepsilon)}\right]
\end{eqnarray*}
The first term in the sum above is less than $$e^{- t \, \bar
G(|x|/2)/2} \, \int f(y) \, p_t^D(x,y) \, dy \, .$$ For the second
term we shall assume that the support of $f$ is included in the
ball $B(x,\varepsilon/2)$. Recall that for a brownian motion
starting at $x$, the exit distribution from $B(x,\varepsilon)$ is
uniform on the sphere $S(x,\varepsilon)$. Hence
\begin{eqnarray*}
\E^{\W_x} \left[ f(\omega(t)) \, \BBone_{T_{\partial D} > t \geq
\tau_x(\varepsilon)}\right] & \leq & \E^{\W_x} \left[ f(\omega(t))
\, \BBone_{t \geq \tau_x(\varepsilon)}\right]\\ & \leq & \int f(y)
\,  \E^{\W_x}\left[ \BBone_{t \geq \tau_x(\varepsilon)} \,
\left(\int_{S(x,\varepsilon)} \gamma_{t -
\tau_x(\varepsilon)}(z,y)  d_Sz\right) \right] \, dy
\end{eqnarray*}
where $\gamma$ is the ordinary heat kernel. Since
$|z-y|>\varepsilon/2$ in the above formula,
\begin{eqnarray*}
\gamma_u(z,y) & = & (2\pi u)^{-d/2} \, e^{|z-y|^2/2u} \\ & \leq &
(2\pi u)^{-d/2} \, e^{- \varepsilon^2/8u} \leq e^{-d/2} \,
\left(\frac{\pi \, \varepsilon^2}{2d}\right)^{-d/2} \, ,
\end{eqnarray*}
the latter inequality being obtained by an easy optimization in
$u$. Since $|x|\geq \alpha$ this quantity is bounded on $D$ by
some constant $B$. But $$\E^{\W_x}\left[ \BBone_{t \geq
\tau_x(\varepsilon)}\right] \leq K e^{- \varepsilon^2/8td}$$ for
some constant $K$ depending on $d$ only. So $$\E^{\W_x} \left[
f(\omega(t)) \, \BBone_{T_{\partial D} > t \geq
\tau_x(\varepsilon)}\right] \leq B \, K \, e^{- |x|^2/32td} \,
\int f(y) \, dy \, $$ and gathering all the previous results we
obtain
\begin{equation}\label{eqboundproche}
\textrm{ if } |x-y|\leq |x|/4 \, ; \, \tilde p_t(x,y) \, \leq  \,
M \, \left((2\pi t)^{-d/2} \, e^{- t \, \bar G(|x|/2)/2} \, + \,
e^{- |x|^2/32td}\right) \, ,
\end{equation}
for some constant $M$. If $|x-y|\geq |x|/4$, \eqref{eqcontrolheat}
furnishes $$\tilde p_t(x,y) \leq \,  e^{Ct/2} \, (2\pi t)^{-d/2}
\, e^{-|x|^2/32t} \, .$$ Define
\begin{equation}\label{eqspeed}
B(u) = \frac 12 \, \min \, (\bar G(u/2)/2 \, , \, (u^2/32 d)) \, .
\end{equation}
We have shown that there exists some constant $M$ such that
$\tilde p_1(x,y) \, \leq \, M \, e^{-2B(|x|)}$, but since $\tilde
p_1$ is symmetric the same holds replacing $|x|$ by $|y|$ and
finally $|x|$ by $\max (|x|,|y|)$. Taking the geometric average of
this estimate and \eqref{eqcontrolheat} ends the proof.
\end{proof}

\begin{proof} of Lemma \ref{lemultra}: let us consider a bounded
function $\|g\|_\infty\leq 1$. Then $P_1 g(x) = e^{V(x)} \, \tilde
P_1(e^{-V} g)(x)$. According to \eqref{eql1a} and lemma
\ref{lembound},
$$\int_{D \cap \{d(x,\partial D)\leq R\}} |\tilde P_1(e^{-V}
g)(x)| dx  \, \leq
$$
\begin{eqnarray*}
 & \leq &
\int_ {D \cap \{d(x,\partial D)\leq
R\}}\left(\int_{D\cap \{d(y,\partial D)>R\}} e^{-V(y)} \tilde
p_1(x,y) dy\right)dx \\& &  + \, M \, \int_{D \cap \{d(x,\partial D)\leq R\}}
\left(\int_{D\cap \{d(y,\partial D)\leq R\}} e^{-V(y)}
 p^D_1(x,y) dy\right) dx \\
& \leq & \, M \, \left(\int_{D\cap \{d(y,\partial D)>R\}} e^{-2V(y)} dy\right)^{1/2} \, \times
\\ & & \, \left(\int_{D\cap
\{d(y,\partial D)>R\}} \left(\int_{D \cap \{d(x,\partial D)\leq R\}} \,
 e^{-|x-y|^2/4} \, e^{-B(|x|\vee |y|)} \, dx\right)^2
\, dy\right)^{1/2}\\
& & \, + \, M \, \int_{D\cap \{d(y,\partial D)\leq R\}} e^{-V(y)} \left( \int_D  p_1^D(x,y) \, dx\right)dy
\end{eqnarray*}
is finite (recall that $p_1^D(x,y)=p_1^D(y,x)$). Hence
$\BBone_{d(x,\partial D)\leq R} \, \tilde P_1(e^{-V} g) \in
L^1(dx)$. Since $\tilde P_t$ is ultracontractive, $\tilde P_1$ is
a bounded map from $\L^1(dx)$ to $\L^2(dx)$. It follows that
$\tilde P_1(\BBone_{d(x,\partial D)\leq R} \, \tilde P_1(e^{-V}
g)) \in L^2(dx)$.

In addition $$ \int_{D \cap \{d(x,\partial D)> R\}} |\tilde P_1(e^{-V} g)(x)|^2 dx \leq
\int_{D \cap \{d(x,\partial D)>
R\}} e^{-2V(x)} dx < +\infty \, ,$$ i.e. $\BBone_{d(x,\partial D)> R} \, \tilde P_1(e^{-V} g) \in L^2(dx)$ so
that $\tilde P_1(\BBone_{d(x,\partial D)> R} \, \tilde P_1(e^{-V} g)) \in L^2(dx)$.

Summing up yields that $\tilde P_2(e^{-V} g) \in \L^2(dx)$.

Actually we may replace $1$ by $s>0$ as remarked in remark \ref{remtemps}, thus replace $2$ by $1$
in the previous result.
\end{proof}

\subsection{\bf Uniqueness of the quasi-stationary distribution.}

In \cite{Cat07} Theorem 7.2 we derived a necessary and sufficient
condition for $\nu_1$ to be the only quasi-limiting distribution,
i.e. to satisfy $\lim_{t \to +\infty} \P_{\nu}(X_t \in A \, | \,
T_{\partial D}> t) = \nu_1(A)$ for all initial distribution $\nu$.
In that case   $\nu_1$ is the only quasi-stationary distribution.
(Recall that in \cite{Cat07}, $D=\R^+$). This condition is very
close to the ultracontractivity of  the semi-group $P_t$.

We shall not try here to obtain such a criterion, but only a
sufficient condition based on the previous remark.

\begin{proposition}\label{propultraunique}
Assume that (H) and \eqref{eql1a} are satisfied. If $P_t$ is an ultracontractive semi-group, for
all initial distribution $\nu$ and all Borel subset $A\subset D$
$$\lim_{t \to +\infty} \P_{\nu}(X_t \in A \, | \, T_{\partial D}> t) = \nu_1(A) \, .$$
In particular $\nu_1$ is the unique quasi-stationary distribution.
\end{proposition}
\begin{proof}
Since $P_t$ is ultracontractive, it turns out, as proved in
\cite{Dav} Theorem 1.4.1,  that $\L^1(\mu)\cap L^\infty(\mu)$
which is included into $\L^2(\mu)$, is invariant under $P_t$. So
$P_t$ extends as a contraction semi-group on all $\L^p(\mu)$.

Now, according to Theorem \ref{thmstructure} we know that
$\P_\nu(X_t \in A \, , \, T_{\partial D}> t) = \int_A \,
r_\nu(t,y) \mu(dy)$ with $r_\nu(t,.)=\int r(t,x,.) \nu(dx) \in
\L^1(\mu)$ and $r_\nu(t+s,y)= P_s (r_\nu(t,.))(y)$. Hence for
$t>2$,
\begin{eqnarray*}
\P_\nu(X_t \in A \, , \, T_{\partial D}
> t) & = & \int_D \, \BBone_A \, P_{t-1}(r_\nu(1,.)) \, d\mu \\ & = & \int_D \, P_1(\BBone_A) \,
P_{t-2}(r_\nu(1,.)) \, d\mu \, .
\end{eqnarray*}
But thanks to Lemma \ref{lemultra} and to ultracontractivity, both
$P_1(\BBone_A)$ and $P_{t-2}(r_\nu(1,.))$ are in $\L^2(\mu)$.
Furthermore $e^{\lambda_1(t-3)} \, P_{t-2}(r_\nu(1,.))$ converges
strongly in $\L^2$ to $\langle P_1(r_\nu(1,.)),\eta_1\rangle \,
\eta_1$ as $t \to +\infty$. Hence $$\lim_{t \to +\infty}
\P_\nu(X_t \in A \, | \, T_{\partial D}> t) = \frac{\langle
P_1(\BBone_A), \eta_1\rangle}{\langle P_1(\BBone_D),
\eta_1\rangle} =\nu_1(A) \, $$ since $\eta_1 \in \L^1(\mu)$.
\end{proof}

It remains to give tractable conditions for $P_t$ to be ultracontractive. To this end we shall use
the ideas introduced in \cite{KKR} and later developed in \cite{cat5} in particular. The next
lemma is the key

\begin{lemma}\label{lemkkr}{\bf (see \cite{KKR})}
Assume that (H) and \eqref{eql1a} are satisfied. If for all $t>0$
there exists $c(t)$ such that for all $x \in D$,
\begin{equation}\label{eqkkr}
e^{V(x)} \, \E^{\W_x}\left[\BBone_{T_{\partial D}>t} \, e^{- \frac 12 \, \int_0^t G(\omega_s)
ds}\right] \, \leq \, c(t) \, ,
\end{equation}
then $P_t$ is ultracontractive. (Recall that $G(y)=|\nabla
V|^2(y)-\triangle V(y)$).

Conversely this condition is necessary if we assume in addition that $\int_D e^{V(x)} \mu(dx) <
+\infty$.
\end{lemma}
\begin{proof}
The proof is the same as in \cite{KKR}. It is given for the sake of completeness.

Recall that the heat semi-group on $D$ is ultracontractive, i.e.
for all non-negative $f \in \L^2(dx)$, $$\sup_{x \in D}
\E^{\W_x}\left[\BBone_{T_{\partial D}>t} \, f(\omega_t)\right]
\leq (\pi t)^{- \frac 14} \, \parallel f \parallel_{\L^2(dx)} \,
.$$ For a non-negative $g\in \L^2(d\mu)$, $f=e^{V}g \in \L^2(dx)$
so that using Theorem \ref{thmstructure}
\begin{eqnarray*}
P_t g(x) & = & e^{V(x)} \, \E^{\W_x}\left[\BBone_{T_{\partial D}>t} \, f(\omega_t) \, e^{- \frac
12 \, \int_0^t G(\omega_s) ds}\right]\\ & \leq & e^{V(x)} \, \E^{\W_x}\left[\BBone_{T_{\partial
D}>t/2} \, e^{- \frac 12 \, \int_0^{t/2} G(\omega_s) ds} \,
\E^{\W_{\omega_{t/2}}}\left[\BBone_{T_{\partial D}>t/2} \, f(\omega'_{t/2}) \, e^{- \frac 12 \,
\int_0^{t/2} G(\omega'_s) ds}\right]\right]\\ & \leq & e^{\frac{Ct}{4}} \,  (\pi t)^{- \frac 14}
\,
\parallel g \parallel_{\L^2(d\mu)} \, e^{V(x)} \, \E^{\W_x}\left[\BBone_{T_{\partial D}>t/2} \, e^{- \frac
12 \, \int_0^{t/2} G(\omega_s) ds}\right] \, .
\end{eqnarray*}
Hence if \eqref{eqkkr} is satisfied, $$P_t g(x) \leq c(t/2) e^{\frac{Ct}{4}} \,  (\pi t)^{- \frac
14} \,
\parallel g \parallel_{\L^2(d\mu)}$$ for all $x \in D$, i.e. $P_t g$ is bounded and $P_t$ is
ultracontractive. \eqref{eqkkr} is thus a sufficient condition. It
is also necessary once $e^{V} \in \L^1(\mu)$, since $$P_t
(e^{V})(x) \, = \, e^{V(x)} \, \E^{\W_x}\left[\BBone_{T_{\partial
D}>t} \, e^{- \frac 12 \, \int_0^t G(\omega_s) ds}\right] \, ,$$
as it can be observed in the Girsanov formula stated in Theorem
\ref{thmstructure}.
\end{proof}

Papers \cite{KKR} and \cite{cat5} contain several methods to prove
\eqref{eqkkr}. The most adapted one to our situation is the ``
well method'' based on the Girsanov transform (Theorem
\ref{thmstructure}).

To this end we shall introduce some notation :
\begin{itemize}
\item for $R>0$, $A_R = D \cap \{|x|\leq R\}$, \item $\bar V(R) = \sup_{x \in A_R} V(x)$ (be
careful that there is no absolute value), \item for $\varepsilon > 0$, $D_\varepsilon = \{y \in D
\, , \, d(y,\partial D)>\varepsilon\}$ and $S_\varepsilon = \inf \{t>0 \, , \, \omega(t) \in
D_\varepsilon^c \}$, \item for $k\in \N^*$ , $e_k = \inf \{t>0 \, , \, \omega(t) \in A_k \}$.
\end{itemize}
We then have the following analogue of Theorem 3.3 in \cite{KKR}

\begin{proposition}\label{propkkr}
Assume that \eqref{eqh1}, (H2) and (H3) are satisfied. We shall
also assume that for all $\varepsilon > 0$, $V$ is bounded from
below on $D_\varepsilon$. Let $a_k= \bar G(k)$, $b_k=\bar V(k)$
and $\gamma_k$ be such that $\sum_{k=1}^\infty \, \gamma_k <
+\infty$. $P_t$ is ultracontractive as soon as the following
holds: \be \label{ultra} \textrm{ for all } \beta > 0 \, , \,
\sum_{k=1}^\infty \exp \left( \frac 12 \, \left(b_{k+1} - \beta \,
\gamma_k \, a_k \right)\right) \, < \, + \infty \, .\ee
\end{proposition}
\begin{proof}
The first step is to be convinced that Lemma 3.1 in \cite{KKR} is
still true i.e if $\tau_R = \inf \{t>0 \, , \, \omega_t \in A_R\}$
and $x \notin A_R$ , $$e^{V(x)} \,
\E^{\W_x}\left[\BBone_{T_{\partial D}>\tau_R} \, e^{- \frac 12 \,
\int_0^{\tau_R} G(\omega_s) ds}\right] \leq e^{\bar V(R)} \, .$$
Define $M_t = e^{-V(\omega_t) - \frac 12 \, \int_0^t G(\omega_s)
ds}$. Thanks to our hypothesis on $V$ and (H2), $M_{t\wedge \tau_R
\wedge S_\varepsilon}$ is actually a bounded martingale. The
result follows by making successively $t$ go to infinity and
$\varepsilon$ go to 0.

Once this is proved the rest of the proof is exactly the same as in \cite{KKR} except that we have
to replace the stopping times $\tau_j$ therein by $e_j \wedge S_\varepsilon$ and then make
$\varepsilon$ go to $0$ again.
\end{proof}

\bigskip

\bibliographystyle{plain}

\end{document}